\documentclass[11pt]{amsart}
\usepackage{txfonts}
\usepackage{amsmath,amsthm,amscd,amssymb,mathtools}
\usepackage{graphicx}
\usepackage[colorlinks,plainpages,backref,urlcolor=blue]{hyperref}
\usepackage{tikz}
\usetikzlibrary{cd, calc, arrows}
\usepackage{enumitem}

\numberwithin{equation}{section}

\newtheorem{theorem}{Theorem}[section]
\newtheorem{corollary}[theorem]{Corollary}
\newtheorem{lemma}[theorem]{Lemma}
\newtheorem{prop}[theorem]{Proposition}

\theoremstyle{definition}
\newtheorem{definition}[theorem]{Definition}
\newtheorem{example}[theorem]{Example}
\newtheorem{remark}[theorem]{Remark}
\newtheorem{question}[theorem]{Question}
\newtheorem{problem}[theorem]{Problem}

\newtheorem{cons}{Construction}

\newcommand{\A}{\mathcal{A}}

\newcommand{\M}{\mathcal{M}}
\newcommand{\VV}{\mathcal{V}}
\newcommand{\RR}{\mathcal{R}}

\newcommand{\ZZ}{\mathcal{Z}}

\newcommand{\C}{\mathbb{C}}
\newcommand{\Z}{\mathbb{Z}}

\newcommand{\Q}{\mathbb{Q}}

\newcommand{\CP}{\mathbb{CP}}
\renewcommand{\P}{\mathbb{P}}
\renewcommand{\k}{\Bbbk}

\newcommand{\bm}{\mathbf{m}}

\DeclareMathOperator{\ii}{i}

\newcommand{\codim}{\operatorname{codim}}

\newcommand{\sd}{\operatorname{sd}}

\newcommand{\lk}{\operatorname{lk}}
\newcommand{\Sing}{\operatorname{Sing}}
\newcommand{\id}{\operatorname{id}}

\newcommand{\ZK}{\mathcal{Z}_K}
\newcommand{\MM}{\mathsf{M}}

\newcommand{\surj}{\twoheadrightarrow}
\newcommand{\inj}{\hookrightarrow}
\newcommand\isom{\xrightarrow{ \,\smash{\raisebox{-0.6ex}{\ensuremath{\scriptstyle\simeq}}}\,}}

\definecolor{lime}{HTML}{A6CE39}
\DeclareRobustCommand{\orcidicon}{
	\begin{tikzpicture}
	\draw[lime, fill=lime] (0,0) 
	circle [radius=0.16] 
	node[white] {{\fontfamily{qag}\selectfont \tiny ID}};
	\draw[white, fill=white] (-0.0625,0.095) 
	circle [radius=0.007];
	\end{tikzpicture}
	\hspace{-2mm}
}
\foreach \x in {A, ..., Z}{\expandafter\xdef\csname orcid\x\endcsname{\noexpand\href{%
https://orcid.org/\csname orcidauthor\x\endcsname}{\noexpand\orcidicon}}}

\title[Non-formal Milnor fibers via polyhedral products]%
{Highly connected non-formal Milnor fibers\\ via polyhedral products}

\author[Alexander~I.~Suciu]{Alexander~I.~Suciu$^1$\!\!\orcidA{}}
\address{Department of Mathematics,
Northeastern University,
Boston, MA 02115, USA}
\email{\href{mailto:a.suciu@northeastern.edu}{a.suciu@northeastern.edu}}
\urladdr{\href{https://suciu.sites.northeastern.edu}%
{suciu.sites.northeastern.edu}}
\thanks{Partially supported by the project ``Singularities and Applications'' - CF 132/31.07.2023 funded by the European Union - NextGenerationEU$^1$ - through Romania's National Recovery and Resilience Plan.}



\begin{document}

\begin{abstract}
We show that the realization theorem of Fern\'andez de Bobadilla,
which identifies the Milnor fiber of a weighted-homogeneous 
polynomial with the complement of a germ of an analytic set, can be
combined with the systematic Massey product constructions of
Grbi\'c--Linton for moment-angle complexes $\ZK = \ZK(D^2, S^1)$
to produce weighted-homogeneous polynomials whose Milnor 
fibers are arbitrarily highly connected and non-formal.
The original application of this strategy, due to Fern\'andez de
Bobadilla, used the Denham--Suciu classification of lowest-degree
triple Massey products and yielded only $2$-connected
non-formal Milnor fibers. The Grbi\'c--Linton framework, which
constructs non-trivial $n$-fold Massey products in $H^*(\ZK;\Q)$ for
arbitrary $n$ and in arbitrary cohomological degrees, removes this
connectivity restriction entirely. For the explicit families we construct, 
we also compute the dimension of the singular set and show that the 
resulting Kato--Matsumoto connectivity bound is sharp.
\end{abstract}

\maketitle

\section{Introduction}
\label{sec:intro}

The formality question for Milnor fibers of hypersurface singularities
was raised by Papadima--Suciu~\cite{PS-formal}.  Recall that a
topological space $X$ is \emph{formal} (over $\Q$) if its rational
homotopy type is determined by its cohomology ring; it is
\emph{$q$-formal} if the $q$-minimal Sullivan model of $X$ agrees
with that of $H^*(X;\Q)$.  The complement $M(\A)$ of a hyperplane
arrangement is always formal, by Brieskorn's theorem \cite{Br}; the
question is whether the same holds for its Milnor fiber $F(\A)$.

Two non-formal families are known.  Zuber~\cite{Zu} showed that the
Milnor fiber of the Ceva arrangement $\A(3,3,3)$ is not $1$-formal,
the obstruction arising from an irrational pencil on the fiber rather
than from a Massey product; the fiber is not simply connected.  A broader
family, constructed in~\cite{Su-mfnf}, uses arrangements supporting
reduced $3$-multinets; these fibers are likewise not simply connected.
Fern\'andez de Bobadilla~\cite{FdB}, building on work of Denham--Suciu~\cite{DS07},
found the first non-formal Milnor fiber that \emph{is} simply connected:
a weighted-homogeneous polynomial in $\C^{11}$ whose Milnor fiber is
homotopy equivalent to a moment-angle complex $\ZK$, with non-formality
witnessed by a non-trivial triple Massey product in $H^8(F;\Z)$.
All known examples in this family, however, share the same connectivity
profile (simply connected), because the underlying Denham--Suciu
construction is confined to simplicial complexes with edges as minimal
non-faces, giving only $2$-connected moment-angle complexes.

The purpose of this paper is to show that the work of
Grbi\'c--Linton~\cite{GL} removes this restriction completely.
Their systematic construction of non-trivial $n$-fold Massey products
in $H^*(\ZK;\Z)$, via star deletions on joins of simplicial complexes,
produces moment-angle complexes with arbitrarily prescribed connectivity
and non-formality.  Combined with the Fern\'andez de Bobadilla realization 
theorem, this yields our main result (Theorem~\ref{thm:main}): for every 
$k\ge 1$ and $n\ge 3$ there exists a weighted-homogeneous polynomial
$\Phi_K\colon\C^N\to\C$ with trivial geometric monodromy 
whose Milnor fiber $F_K=\Phi_K^{-1}(1)$ is $(2k+2)$-connected
and non-formal, with non-formality witnessed by a non-trivial
$n$-fold Massey product in $H^{2n(k+1)+2}(F_K;\Q)$.

\medskip\noindent\textbf{A first example.}
To give a sense of what the polynomials of Theorem~\ref{thm:main} look like,
consider the first member of the family of Theorem~\ref{thm:family}
($k=1$, $n=3$, $q=2$): take three disjoint triangles $K^i=\partial\Delta^2$ and
apply two Grbi\'c--Linton star deletions to their join (see
Example~\ref{ex:k2n3}). The resulting weighted-homogeneous polynomial
$\Phi_K\colon\C^{31}\to\C$ has degree $5$ and $22$ generators: $21$ degree-$3$
monomials (each a product of three variables, one from each $K^i$, of $y$-weight
$2$) and a single degree-$4$ monomial (of $y$-weight $1$). Its Milnor fiber
$F_K=\Phi_K^{-1}(1)\simeq\ZK$ has trivial geometric monodromy and is
$4$-connected. Moreover, $F_K$ is $12$-formal but not $13$-formal, and the
non-formality is witnessed by a non-trivial triple Massey product in
$H^{14}(F_K;\Q)\cong\Q$ -- the lowest possible degree given the connectivity.
For general $k\ge 1$ and $n=3$, the analogous family (Theorem~\ref{thm:family})
has exact formality degree $6k+6$, with the Massey product in $H^{6k+8}(F_K;\Q)$.

\medskip\noindent\textbf{The role of the singular set.}
A complementary perspective is provided by the Kato--Matsumoto
theorem~\cite{KM75}: for a holomorphic function germ
$f\colon(\C^N,0)\to(\C,0)$, the Milnor fiber is at least
$(N-s-2)$-connected, where $s=\dim_0\Sing(f)$.
For the polynomials $\Phi_K$ of Theorem~\ref{thm:main}, the 
non-degeneracy condition of Proposition~\ref{prop:sing} holds throughout 
the family of Theorem~\ref{thm:family} (with $k\ge 1$, $n=3$, 
and $q=k+1$): one has $m=3(q+1)$, $\nu=q+1$, and 
\[
  s = m+r-2\nu = m+r-2(q+1),
\]
where $r$ is the number of facets of $K$ (see 
Remark~\ref{rem:family-m2} for explicit values).  The 
Kato--Matsumoto bound 
\[
  N - s - 2 = 2\nu - 2 = 2q
\]
is therefore independent of $r$, and equals the true connectivity 
$2q$ of $F_K$ for all $q \ge 2$: the bound is sharp throughout the 
family.

\medskip\noindent\textbf{Context: partial formality and connectivity.}
The connection between connectivity and partial formality is
transparent: a $c$-connected space satisfies $b_i=0$ for $1\le i\le c$,
so its $c$-minimal Sullivan model is free (no relations), hence trivially
$c$-formal.  The first cohomology appears in degree $c+1$, but a
non-trivial triple Massey product of such classes lands in degree
$3(c+1)-1 = 3c+2$, and by Lemma~\ref{lem:massey-formal} obstructs
$(3c+1)$-formality rather than $(c+1)$-formality; non-formality thus
surfaces far above the connectivity range---a gap made precise by
Theorem~\ref{thm:family}.
The Dimca--Papadima criterion
\cite{DP-pisa} predicts formality when the monodromy acts trivially
on rational homology and the fiber is a cover of a formal base; our fibers have
trivial monodromy but are not covers of any formal space (being simply
connected, they have no non-trivial covers at all), so the criterion
says nothing in either direction.

\medskip
The paper is organized as follows.
Section~\ref{sec:milnor} reviews weighted-homogeneous Milnor fibrations
(\S\ref{subsec:mf}), formality and Massey products
(\S\ref{subsec:formality}), the Dimca--Papadima criterion for finite
group actions with a self-contained proof
(\S\ref{subsec:group-formality}), the known
non-formal Milnor fibers of hyperplane arrangements
(\S\ref{subsec:formality-mf-arr}), the Fern\'andez de Bobadilla
realization theorem (\S\ref{subsec:FdB}), and the Kato--Matsumoto
(KM) connectivity bound (\S\ref{subsec:KM}).
Section~\ref{sec:polyhedral} introduces moment-angle complexes
(\S\ref{subsec:mac}) and their coordinate subspace arrangements
(\S\ref{subsec:coordinate-subspaces}), the associated Milnor
fibration (\S\ref{subsec:milnor-K}), the Hochster formula and
connectivity (\S\ref{subsec:Hochster-conn}), Massey products from
Baskakov to Denham--Suciu (\S\ref{subsec:massey-zk}), and the
Grbi\'c--Linton constructions (\S\ref{subsec:GL}).
Section~\ref{sec:main} states and proves Theorem~\ref{thm:main}.
Section~\ref{sec:examples} works out explicit examples, establishes
the infinite family of Theorem~\ref{thm:family}, and compares the
KM bounds. Section~\ref{sec:questions} collects open problems.

\section{Milnor fibrations and formality}
\label{sec:milnor}

\subsection{The Milnor fibration}
\label{subsec:mf}

Milnor's fibration theorem \cite{Milnor} applies in considerable
generality.  Let $f\in\C[z_0,\dots,z_d]$ be a weighted-homogeneous
polynomial of degree $n$ with positive integer weights
$(w_0,\dots,w_d)$, and let $M=\C^{d+1}\setminus V(f)$ be the
complement of its zero-set.  Then the restriction
\[
  f\colon M \longrightarrow \C^*
\]
is the projection of a smooth, locally trivial fiber bundle---the
\emph{(global) Milnor fibration} of $f$---with typical fiber
the \emph{Milnor fiber} $F = F(f) = f^{-1}(1)$.  The space $F$
is a smooth affine variety having the homotopy type of a finite
CW-complex of dimension at most~$d$.  Assuming $f$ is not a proper
power, $F$ is connected.  The \emph{geometric monodromy} is the
diffeomorphism
\[
  h\colon F\longrightarrow F, \quad
  (z_0,\dots,z_d) \longmapsto
  (\zeta_n^{w_0}z_0,\dots,\zeta_n^{w_d}z_d),
\]
where $\zeta_n = e^{2\pi\ii/n}$.  The map $h$ has order $n$ and
generates a $\Z_n$-action on $F$, free off a subvariety of positive
codimension; the quotient $F\to F/\Z_n$ is therefore a ramified
cyclic cover, unramified in the homogeneous case treated below.
The mapping torus of $h$ is homotopy equivalent to $M$.
See \cite{MilnorOrlik} for an early study of the topology of $F$ in
this generality.

The induced maps $h_i = h_*\colon H_i(F;\Q)\to H_i(F;\Q)$
are the \emph{algebraic monodromy operators}; since $h^n = \id$,
each $h_i$ is diagonalizable with eigenvalues that are $n$-th roots
of unity.  Their characteristic polynomials $\Delta_i(t) = \det(t\,\id- h_i)$
record the monodromy data of the fibration;
the $1$-eigenspace $\ker(h_i - \id)$ has dimension $b_i(F/\Z_n)$.

This bundle is trivial when $f$ has a linear factor: a linear form $\ell$ not
vanishing on $M$ gives the splitting $z \mapsto (z,\ell(z))$, whence
$M \cong U \times \C^*$; in particular this holds for a hyperplane arrangement,
$\ell$ being any one of its defining forms.

Independently of any such splitting, homogeneity makes the monodromy
$h(z)=\zeta_n z$ scalar multiplication by $\zeta_n$, so $F/\Z_n\cong U$ and
$\sigma\colon F\to U$ is a regular cyclic $n$-fold cover.  In the
weighted-homogeneous case with unequal weights the monodromy is instead a
diagonal, non-scalar automorphism of $F$, and no such projectivized splitting
is available in general.

\subsection{Formality and Massey products}
\label{subsec:formality}

A path-connected space $X$ is \emph{$q$-formal} (over $\Q$) if its
$q$-minimal Sullivan model is weakly $q$-equivalent to
$(H^*(X;\Q),0)$; full formality corresponds to $q=\infty$, and
$q$-formality implies $r$-formality for all $r\le q$
(see e.g.~\cite{PS-formal}).

\begin{lemma}
\label{lem:vanishing-formal}
If $b_i(X)=0$ for all $1\le i\le q$, then $X$ is $q$-formal.
\end{lemma}

\begin{proof}
The $q$-minimal Sullivan model of $X$ has no generators in degrees
$1,\dots,q$ (since $H^i(X;\Q)=0$ for those degrees), hence no
relations in degrees $\le q+1$, and is automatically isomorphic
to the $q$-minimal model of $(H^*(X;\Q),0)$.
\end{proof}

We recall the higher operations that detect non-formality \cite{DGMS,Sullivan77}.

\begin{definition}
\label{def:massey}
Let $(A,d)$ be a differential graded algebra over a commutative ring,
and let $a_1,\dots,a_k\in H^*(A)$ have degrees $\deg a_i = n_i$.  A
\emph{defining system} for $\langle a_1,\dots,a_k\rangle$ is a family
of elements $a_{i,j}\in A$, one for each pair $1\le i\le j\le k$ with
$(i,j)\ne(1,k)$, of degree $n_i+\cdots+n_j-(j-i)$, such that each
$a_{i,i}$ is a cocycle representing $a_i$ and
\begin{equation}
\label{eq:defining-system}
  d\,a_{i,j}=\sum_{l=i}^{j-1}\bar a_{i,l}\,a_{l+1,j},
  \qquad \bar a=(-1)^{\deg a}a .
\end{equation}
When a defining system exists, the cocycle
$\sum_{l=1}^{k-1}\bar a_{1,l}\,a_{l+1,k}$ has degree
$L=n_1+\cdots+n_k-(k-2)$, and its class is the corresponding
\emph{value}.  The $k$-fold \emph{Massey product}
$\langle a_1,\dots,a_k\rangle\subseteq H^{L}(A)$ is the set of all
values; it is a coset of an indeterminacy subgroup.  It is
\emph{defined} when a defining system exists and \emph{trivial} when
$0\in\langle a_1,\dots,a_k\rangle$.  For $k=3$ this reads
$\langle a_1,a_2,a_3\rangle=[\,\bar a_1\,a_{2,3}+\bar a_{1,2}\,a_3\,]$
with $d\,a_{1,2}=\bar a_1a_2$ and $d\,a_{2,3}=\bar a_2a_3$.
\end{definition}

The level of formality obstructed by a non-trivial Massey product is
governed by the degree in which it \emph{lands}, not by the degrees of
its defining classes; for $H^1$-classes the two agree---a non-trivial
triple product lands in degree~$2$ and obstructs $1$-formality---but
they diverge in higher degrees, as the next lemma makes precise.

\begin{lemma}
\label{lem:massey-formal}
Let $X$ be a simply-connected space of finite $\Q$-type and let
$\langle a_1,\dots,a_k\rangle$ be a defined Massey product of classes
$a_i\in H^{n_i}(X;\Q)$ with all $n_i\ge 2$, landing in degree
$L=n_1+\cdots+n_k-(k-2)$.  If $X$ is $(L-1)$-formal, then
$\langle a_1,\dots,a_k\rangle$ is trivial.  Equivalently, if $X$ is
$q$-formal then every defined Massey product landing in degree
$\le q+1$ is trivial; a non-trivial Massey product landing in degree
$L$ obstructs $(L-1)$-formality.
\end{lemma}

\begin{proof}
Let $\varphi\colon\M=(\Lambda V,d)\isom A_{\mathrm{PL}}(X)$
be the minimal model and $\M_q=(\Lambda V^{\le q},d)$ its
$q$-minimal model.  As $X$ is simply connected,
$d(V^j)\subseteq\Lambda^{\ge 2}(V^{\le j-1})$, so $\M_q$ is a
sub-cdga whose inclusion into $\M$ is an isomorphism on
$H^{\le q}$ and a monomorphism on $H^{q+1}$ \cite{PS-formal}.  Set
$q=L-1$.

Since $\varphi$ is a quasi-isomorphism, $\langle a_1,\dots,a_k\rangle$
may be computed in $\M$.  Each entry of a defining system has
degree $\deg a_{i,j}=1+\sum_{l=i}^{j}(n_l-1)$; for $(i,j)\ne(1,k)$ the
missing index lowers this by $n_l-1\ge 1$ below
$L=1+\sum_{l=1}^{k}(n_l-1)$, so $\deg a_{i,j}\le L-1=q$.  An element of
$\M$ of degree $\le q$ lies in $\Lambda V^{\le q}=\M_q$,
so the entire defining system and a representing cocycle $\Phi$ of
degree $L=q+1$ lie in $\M_q$; thus $[\Phi]\in H^{q+1}(\M_q)$ 
represents a value of the product under the inclusion.

If $X$ is $(L-1)$-formal there is a $q$-equivalence
$\psi\colon\M_q\to(H^*(X;\Q),0)$.  Every Massey product in a
cdga with zero differential contains $0$, so $\psi_*[\Phi]$ lies in the
indeterminacy of $\langle a_1,\dots,a_k\rangle$ in $(H,0)$---a subspace
assembled from cohomology in degrees $\le q$, on which $\psi_*$ is an
isomorphism.  As $\psi_*$ is also injective on $H^{q+1}$, the class
$[\Phi]$ already lies in the indeterminacy inside $\M_q$;
transporting along the inclusion (an isomorphism below degree $q+1$ and
injective in degree $q+1$) shows that $\langle a_1,\dots,a_k\rangle$
contains $0$ in $X$.
\end{proof}

\subsection{Finite group actions and formality}
\label{subsec:group-formality}

The following proposition generalizes a result of Dimca and Papadima
\cite{DP-pisa}; notably, the action of $G$ need not be free.

\begin{prop}
\label{prop:DP}
Let $G$ be a finite group acting on a connected $G$-CW complex $Y$,
and let $X=Y/G$. Assume $X$ is $q$-formal over $\Q$.
\begin{enumerate}[itemsep=1pt]
\item \label{f1}
If $G$ acts trivially on $H^i(Y;\Q)$ for all $i\le q$, then $Y$ is
$q$-formal. 

\item \label{f2}
If moreover $Y$ has the homotopy type of a CW-complex of
dimension $\le q+1$, then $Y$ is formal.
\end{enumerate}
\end{prop}

\begin{proof}
Let $p\colon Y\to X$ be the quotient map.
Since $G$ is finite and $\Q$ has characteristic zero, the
transfer homomorphism $\mathrm{tr}\colon H^*(Y;\Q)\to H^*(X;\Q)$
satisfies $\mathrm{tr}\circ p^* = |G|\cdot\mathrm{id}$; this uses
only that $Y$ is a $G$-CW complex, not that the action is free.
Hence $p^*\colon H^i(X;\Q)\isom H^i(Y;\Q)^G$ is an
isomorphism for all~$i$.
Under the hypothesis that $G$ acts trivially on $H^i(Y;\Q)$ for
$i\le q$, the fixed-part inclusion becomes an isomorphism
$p^*\colon H^i(X;\Q)\isom H^i(Y;\Q)$ for $i\le q$,
and a monomorphism in degree $q+1$.

At the CDGA level, the map $p^*\colon A_{\mathrm{PL}}(X;\Q)\to
A_{\mathrm{PL}}(Y;\Q)$ induces a morphism of $q$-minimal models
$\M_q(X)\to\M_q(Y)$ which is a $q$-equivalence (isomorphism on
cohomology through degree~$q$, monomorphism in degree $q+1$),
by the cohomology comparison above and the uniqueness of
$q$-minimal models.
Since $X$ is $q$-formal, $\M_q(X)\cong\M_q(H^*(X;\Q),0)$.
The $q$-equivalence $\M_q(X)\to\M_q(Y)$ therefore gives
\[
  \M_q(Y)\simeq_q \M_q(X) \cong \M_q(H^*(X;\Q),0)
  \cong \M_q(H^*(Y;\Q),0),
\]
where the last isomorphism uses $H^{\le q}(X;\Q)\cong H^{\le q}(Y;\Q)$.
Hence $Y$ is $q$-formal, proving~\eqref{f1}.

For~\eqref{f2}: if $\dim Y\le q+1$, then the $q$-minimal model
$\M_q(Y)$ already contains all generators of the full minimal
model~$\M(Y)$ (there are no generators of degree $>q+1$ in a
$(q+1)$-dimensional complex), so $\M(Y)\cong\M_q(Y)\cong
\M_q(H^*(Y;\Q),0) = \M(H^*(Y;\Q),0)$, and $Y$ is formal 
(see \cite[Prop.~3.4]{Mac}).
\end{proof}

\subsection{Milnor fibers of hyperplane arrangements}
\label{subsec:formality-mf-arr}

Let $\A$ be a central arrangement of $n$ hyperplanes in $\C^{d+1}$ 
($d\ge 1$). The complement $M(\A)=\C^{d+1}\setminus \bigcup_{H\in\A}H$
is a connected, smooth quasi-projective variety. We let 
$U(\A)=\P(M(\A))\subset \CP^d$ be its projectivization. 
Since $M(\A)$ is formal by Brieskorn's theorem and $\C^*$ is formal,
it follows that $U(\A)$ is formal.

For each $H\in\A$, choose a linear form $\alpha_H$ with
$H=\ker(\alpha_H)$, and set $f=\prod_{H\in\A}\alpha_H$.
Then $f$ is homogeneous of degree $n$, and its Milnor fiber
$F(\A)=f^{-1}(1)$ is a connected, finite CW-complex of dimension $d$.
Moreover, $F(\A)$ is a cyclic regular cover of $U(\A)$ with deck group
$\Z/n\Z$.

\begin{corollary}[\cite{DP-pisa}]
\label{cor:DP}
If the algebraic monodromy acts trivially on
$H^{\le q}(F(\A);\Q)$, then $F(\A)$ is $q$-formal.
If this holds for all $i\le d-1$, then $F(\A)$ is formal.
\end{corollary}

\begin{proof}
Triviality of algebraic monodromy is equivalent to triviality of the
deck group action on $H^*(F(\A);\Q)$. Apply
Proposition~\ref{prop:DP} with $Y=F(\A)$ and $X=U(\A)$, using that
$U(\A)$ is formal and the covering is finite cyclic.
\end{proof}

In particular, non-formality of $F(\A)$ implies non-triviality of
the algebraic monodromy in some degree.

The first example of a non-formal Milnor fiber is due to
Zuber~\cite{Zu}: $F(\A(3,3,3))$ is not $1$-formal.  The argument
proceeds as follows: a rational pencil on the complement
$M(\A(3,3,3))$ lifts to an \emph{irrational} pencil on the Milnor
fiber (via a $3$-fold intermediate cover), producing a
$4$-dimensional subtorus of $\VV^1_1(F)$; a mixed Hodge
structure argument then shows this subtorus cannot be of the form
$\exp(L)$ for any $L\subset\RR^1_1(F)$, violating the Tangent
Cone Theorem~\cite{DPS-duke} and forcing non-$1$-formality.
Non-trivial algebraic monodromy ($\Delta_1(t)\ne(t-1)^8$) is a
byproduct.  The fiber is not simply connected.

A family of non-formal Milnor fibers for arrangements
supporting a reduced $3$-multinet is constructed in
\cite{Su-mfnf}; non-formality comes from non-trivial monodromy
detected via the tangent cone to the characteristic variety.
In both families, $F(\A)$ is \emph{not} simply connected.

For a \emph{multi-arrangement} $(\A,\bm)$ with multiplicity
vector $\bm=(m_H)_{H\in\A}\in\Z_{>0}^{\A}$, one considers
$f_{\bm}=\prod_{H\in\A}\alpha_H^{m_H}$, homogeneous of degree
$N=\sum_H m_H$; its Milnor fiber $F_{\bm}(\A)=f_{\bm}^{-1}(1)$
is a regular $\Z_N$-cover of $U(\A)$, with the algebraic monodromy
defined analogously. For decomposable arrangements $\A$ 
(in the sense of \cite{PS-cmh06}), the monodromy acts trivially
on $H_1(F_{\bm}(\A);\Q)$ for any $\bm$ satisfying a mild
condition \cite[Thm.~1.3]{Su-revroum}, and
Corollary~\ref{cor:DP} then implies $1$-formality.

\begin{question}
\label{q:massey-H2}
Is there a hyperplane arrangement $\A$ for which $H^*(F(\A);\Q)$
carries a non-trivial Massey product on classes in $H^2(F(\A);\Q)$?
\end{question}

\subsection{The Fern\'andez de Bobadilla realization theorem}
\label{subsec:FdB}

The known constructions of non-formal Milnor fibers for hyperplane
arrangements rely on non-trivial algebraic monodromy and yield
examples that are not simply connected. This leaves open the
problem of producing non-formal Milnor fibers with trivial
monodromy and higher connectivity.

A different approach, due to Fern\'andez de Bobadilla \cite{FdB}, 
provides such examples by realizing Milnor fibers as complements of
analytic germs, without passing through a covering construction.
The constructions of this paper rest on a realization theorem of
Fern\'andez de Bobadilla, which identifies the Milnor fiber of a
specific family of weighted-homogeneous polynomials with the
complement of an analytic germ.

Let $\mathcal{O}_{\C^n,O}$ denote the ring of germs of holomorphic
functions at the origin of $\C^n$, and let
$I = (f_1,\dots,f_r) \subset \mathcal{O}_{\C^n,O}$ be an ideal.
Define the polynomial
\[
\Phi_I(x_1,\dots,x_n,\, y_1,\dots,y_r) = \sum_{i=1}^r y_i f_i(x_1,\dots,x_n).
\]

\begin{theorem}[\cite{FdB}, Thm.~1]
\label{thm:FdB}
The Milnor fiber of $\Phi_I$ at the origin is homotopy equivalent
to the complement of the germ $V(I)$ in a small ball:
\[
\Phi_I^{-1}(1) \simeq B_\varepsilon \setminus V(I)
\]
for sufficiently small $\varepsilon > 0$.  Moreover, the geometric
monodromy of $\Phi_I$ is trivial.
\end{theorem}

The geometric monodromy is here trivial in the strong sense that it
can be realized by the identity diffeomorphism of $F$, not merely up
to homotopy or homology.  The proof is immediate from
the linearity of $\Phi_I$ in the variables $y_1,\dots,y_r$: the
family of diffeomorphisms $\varphi_\theta(x,y)=(x,e^{2\pi i\theta}y)$
carries the fiber over any $\delta\in\C^*$ to the fiber over
$e^{2\pi i\theta}\delta$, and equals the identity when $\theta=1$
\cite{FdB}.

When the generators $f_i$ of $I$ are weighted-homogeneous with
respect to a common weight vector on $x_1,\dots,x_n$, the
polynomial $\Phi_I$ is itself weighted-homogeneous (by assigning
each $y_i$ the complementary weight).
In this case, the local and global Milnor fibers coincide.
The reason is transparent: the weighted $\C^*$-action
\[
  \lambda\cdot(x,y)
  =\bigl(\lambda^{w(x_j)}x_j,\,\lambda^{w(y_i)}y_i\bigr)
\]
scales $\Phi_I$ by $\lambda^d$ (where $d$ is the weighted degree),
so $V(\Phi_I)$ is $\C^*$-invariant.
The dilation $\lambda\mapsto\lambda^{1/d}$ therefore carries
the local fiber $\Phi_I^{-1}(\delta)\cap B_\varepsilon$
diffeomorphically onto the global fiber $\Phi_I^{-1}(1)$
for any $\delta\ne 0$ and any $\varepsilon>0$
(see also \cite{Oka73}).
Thus $F(\Phi_I)\simeq\C^n\setminus V(I)$.

\subsection{The Kato--Matsumoto connectivity bound}
\label{subsec:KM}

For a holomorphic germ $f\colon(\C^N, 0)\to(\C, 0)$ with non-isolated
singularity, the Kato--Matsumoto theorem gives a lower bound on
the connectivity of the Milnor fiber in terms of the dimension of
the singular locus.

\begin{theorem}[\cite{KM75}]
\label{thm:KM}
Let $f\colon(\C^N,0)\to(\C,0)$ be a holomorphic  germ,
and let $s = \dim_0\Sing(f)$ be the dimension of
the critical set at the origin.  Then the Milnor fiber $F$ is
$(N-s-2)$-connected; in particular,
$\widetilde{H}_j(F;\Z) = 0$ for all $j < N-s-1$.
\end{theorem}

A key feature used throughout this paper is that for a
weighted-homogeneous polynomial $f$, the local and global Milnor
fibers are diffeomorphic (see \S\ref{subsec:FdB}); hence
Theorem~\ref{thm:KM} bounds the connectivity of the global Milnor
fiber in terms of $s=\dim_0\Sing(f)$, computed from the Jacobian
ideal.  In Section~\ref{sec:examples} we carry this out for the
families constructed in this paper, comparing the resulting
Kato--Matsumoto bound with their actual connectivity.

\section{Polyhedral products and Massey products}
\label{sec:polyhedral}

\subsection{Polyhedral products and moment-angle complexes}
\label{subsec:mac}

Let $K$ be a simplicial complex on vertex set $[m]=\{1,\dots,m\}$.
For pairs of spaces $(X_i,A_i)$ with $A_i\subset X_i$, the associated
\emph{polyhedral product} is the subspace
\begin{equation}
\label{eq:zk-xa}
\ZK(\underline{X},\underline{A}) = 
\bigcup_{\sigma\in K}(\underline{X},\underline{A})^{\sigma}
\subset \prod_{i=1}^m X_i,
\end{equation}
where $(\underline{X},\underline{A})^{\sigma} = \prod_{i=1}^{m} Y_i$ with
$Y_i = X_i$ if $i\in\sigma$ and $Y_i = A_i$ if $i\notin\sigma$. If all pairs 
$(X_i,A_i)$ are equal to a fixed pair $(X,A)$, we simply write the 
result as  $\ZK(X,A)$.
We refer to \cite{BP15,DS07,BBCG} for comprehensive treatments.

For joins of simplicial complexes, the construction satisfies
\begin{equation}
\label{eq:joinprod}
\ZZ_{K*K'}(X,A) \cong \ZK(X,A) \times \ZZ_{K'}(X,A);
\end{equation}
see \cite[Lem.~2.1.4]{DS07}.

The central instance for this paper is the pair $(D^2, S^1)$,
where $D^2\subset\C$ is the closed unit disk and $S^1=\partial D^2$.
The resulting space
\begin{equation}
\label{eq:zk}
\ZK = \ZK(D^2,S^1) \subset (D^2)^m \subset \C^m
\end{equation}
is the \emph{moment-angle complex} of $K$.  It is a compact space
of dimension $m + \dim K + 1$, and is a smooth manifold whenever
$K$ is a triangulation of a sphere \cite{BP02}.

\medskip\noindent\textbf{Rational homotopy type.}
The rational homotopy type of the more general polyhedral product
$\ZK(\underline{X},\underline{X'})$, for nilpotent CW-pairs, was determined 
by F\'elix--Tanr\'e \cite{FT09}.  Let $A_i$, $A_i'$ be connected
finite-type rational CDGA models for $X_i$, $X_i'$ with surjective
maps $\varphi_i\colon A_i\surj A_i'$ modeling the
inclusions $X_i'\inj X_i$.  For each simplex $\sigma$
on $[m]$, let $I_\sigma = \prod_{i=1}^m E_i$ with $E_i = \ker(\varphi_i)$
if $i\in\sigma$ and $E_i = A_i$ if $i\notin\sigma$.

\begin{theorem}[\cite{FT09}]
\label{thm:FT}
The polyhedral product $\ZK(\underline{X},\underline{X'})$ 
has a rational CDGA model of the form
\[
  A(K) = \Bigl(\bigotimes_{i=1}^m A_i\Bigr)\Big/I(K),
\]
where $I(K) = \sum_{\sigma\notin K} I_\sigma$.  Moreover,
inclusions of subcomplexes $L\subset K$ are modeled by the
projection $A(K)\surj  A(L)$.
\end{theorem}

The additive structure of $H^*(\ZK(\underline{X},\underline{X'});\Q)$
is given by the Bahri--Bendersky--Cohen--Gitler decomposition
\cite{BBCG}; in the moment-angle case $(D^2,S^1)$ this is the
Hochster formula recalled in \S\ref{subsec:Hochster-conn}.
In particular, Theorem~\ref{thm:FT} provides a CDGA model $A(K)$ for
$\ZK(\underline{X},\underline{X'})$; the question of whether this 
model is formal (i.e., quasi-isomorphic to its cohomology) is closely 
related to the non-formality problems studied in this paper, and to 
the presence of non-trivial Massey products in $H^*(\ZK;\Q)$.

\subsection{Coordinate subspace arrangements}
\label{subsec:coordinate-subspaces}
There is an equivalent description of the moment-angle complex $\ZK$ 
as the complement of a coordinate subspace arrangement.  Define
\begin{equation}
\label{eq:ak-arr}
\A_K = \{ L_\sigma : \sigma\text{ a minimal non-face of }K\},
\quad L_\sigma = \{z\in\C^m : z_i=0 \text{ for all } i\in\sigma\}.
\end{equation}
Let $V(\A_K)=\bigcup_{L\in \A_K} L$ be the union of these 
coordinate subspaces, and $M(\A_K)=\C^m\setminus V(\A_K)$ 
the complement of the subspace arrangement.
A point $z\in\C^m$ lies in $V(\A_K)$ if and only if $\{i:z_i=0\}\notin K$;
equivalently, $z\in M(\A_K)$ if and only if $\{i:z_i=0\}\in K$.  Hence
$M(\A_K) = \ZZ_K(\C,\C^*)$, and \cite[Thm.~4.7.5]{BP15} yields a 
$\mathbb{T}^m$-equivariant deformation retraction
\begin{equation}
\label{eq:mak-zk}
M(\A_K) \simeq \ZK. 
\end{equation}

Recall that the \emph{Stanley--Reisner ideal} of $K$ is the squarefree
monomial ideal
\begin{equation}
\label{eq:sr-ideal}
I_K = \bigl(x_\sigma : \sigma\notin K\bigr) \subseteq \C[x_1,\dots,x_m],
\qquad x_\sigma \coloneqq \prod_{i\in\sigma} x_i,
\end{equation}
minimally generated by the monomials $x_\sigma$ of the minimal
non-faces of $K$.  It governs the cohomology ring of $\ZK$
(\S\ref{subsec:massey-zk}).  The defining ideal of the arrangement
$\A_K$, by contrast, is the Stanley--Reisner ideal of the
\emph{Alexander dual} $K^\vee = \{\tau\subseteq[m] : [m]\setminus\tau\notin K\}$,
as one sees by intersecting the ideals $(x_i:i\in\sigma)$ of the
components $L_\sigma$ of $V(\A_K)$ and applying Alexander duality for
monomial ideals (see, e.g., \cite[\S 2.4]{BP15}):
\begin{equation}
\label{eq:IVAK}
I(V(\A_K)) = \bigcap_{\sigma\,\text{min.~non-face}} (x_i:i\in\sigma) = I_{K^\vee}.
\end{equation}
The minimal generators of $I_{K^\vee}$ are the squarefree monomials
\begin{equation}
\label{eq:gF}
g_F \coloneqq \prod_{i\in[m]\setminus F} x_i, \qquad F\text{ a facet of }K,
\end{equation}
or, equivalently, the monomials $x^T$ where $T$ runs over the minimal 
transversals of the hypergraph of minimal non-faces of $K$. 
Since $I_{K^\vee}$ is a squarefree monomial ideal, hence radical, the
ideal equality~\eqref{eq:IVAK} is equivalent to the equality of zero sets
\begin{equation}
\label{eq:VAK}
V(\A_K) = V(I_{K^\vee}).
\end{equation}

\subsection{The Milnor fibration associated to $K$}
\label{subsec:milnor-K}

We now specialize the Fern\'andez de Bobadilla construction of
\S\ref{subsec:FdB} to the ideal $I = I_{K^\vee}$, whose minimal
generators are the monomials $g_F$ of~\eqref{eq:gF}, of degrees
$d_F = m - |F|$.  Enumerate the facets of $K$ as $F_1,\dots,F_r$,
write $g_j = g_{F_j}$ and $d_j = m - |F_j|$, and set
$d = 1 + \max_j d_j$.  Assigning $w(x_i)=1$ for $i\in[m]$ and
$w(y_j) = d - d_j$ gives positive integer weights for which every
term $y_j g_j$ has weighted degree $d$, so
\begin{equation}
\label{eq:Phi_K}
\Phi_K(x,y) \,:=\, \sum_{j=1}^{r} y_j\, g_j(x)
\end{equation}
is weighted-homogeneous of degree $d$; it is the polynomial $\Phi_I$
of~\S\ref{subsec:FdB} for $I = I_{K^\vee}$.  Moreover, $\Phi_K$ is
homogeneous (with $w(x_i)=w(y_j)=1$ for all $i,j$) if and only if
all $d_j$ are equal, i.e., if and only if $K$ is \emph{pure} (all
facets have the same cardinality).

Since $\C^m\setminus V(I_{K^\vee}) = \C^m\setminus V(\A_K) =
\ZZ_K(\C,\C^*) \simeq \ZK$ by~\eqref{eq:VAK} and the deformation
retraction above, Theorem~\ref{thm:FdB} (in its global,
weighted-homogeneous form) yields:

\begin{corollary}
\label{cor:FdB-ZK}
Let $K$ be a simplicial complex on $[m]$ with facets
$F_1,\dots,F_r$, and let $\Phi_K(x,y) = \sum_{j=1}^r y_j\, g_{F_j}(x)$
be the weighted-homogeneous polynomial defined in~\eqref{eq:Phi_K}.
Then the Milnor fiber $F_K = \Phi_K^{-1}(1)$ is homotopy equivalent
to $\ZK$, and the geometric monodromy of $\Phi_K$ is trivial.
\end{corollary}

Three spaces should be kept apart here, in two different ambient
dimensions.  The Milnor fibration of $\Phi_K$ is
$\Phi_K\colon M\to\C^*$, with total space the complement
$M = \C^N\setminus V(\Phi_K)$ of the \emph{hypersurface}
$V(\Phi_K)\subset\C^N$, where $N = m+r$.  This hypersurface must not
be confused with the coordinate subspace arrangement
$V(\A_K)\subset\C^m$ of~\eqref{eq:VAK}, and $M$ must not be confused
with the arrangement complement $M(\A_K) = \C^m\setminus V(\A_K)$: it
is the Milnor \emph{fiber} $F_K\subset\C^N$, not $M$, that satisfies
$F_K\simeq M(\A_K)\simeq\ZK$ (Corollary~\ref{cor:FdB-ZK}).  As for
$M$ itself, the linearity of $\Phi_K$ in $y_1,\dots,y_r$ gives
$\Phi_K(x,cy) = c\,\Phi_K(x,y)$, so the map
$\bigl((x,y),c\bigr)\mapsto(x,cy)$ is a diffeomorphism
$F_K\times\C^*\isom M$; this both exhibits the triviality of the
geometric monodromy and identifies
$M\cong F_K\times\C^*\simeq\ZK\times\C^*$.

\begin{example}
\label{ex:two-vertices}
The smallest instance of the construction, illustrating
\eqref{eq:VAK} and Corollary~\ref{cor:FdB-ZK}, is the
two-point complex $K$ on $[2]$, whose only non-face is 
$\{1,2\}$, so that $V(\A_K) = L_{\{1,2\}} = \{x_1 = x_2 = 0\}$
is the origin.  Its facets are $\{1\}$ and $\{2\}$, with $g_{\{1\}} = x_2$
and $g_{\{2\}} = x_1$, so $I_{K^\vee} = (x_1, x_2)$ and indeed
$V(I_{K^\vee}) = \{0\} = V(\A_K)$.  The associated polynomial 
$\Phi_K = y_{1} x_2 + y_{2} x_1$ is a nondegenerate quadratic form, 
and its Milnor fiber $F_K = \Phi_K^{-1}(1)$ is homotopy equivalent to 
$S^3 = \ZK(D^2,S^1)$, the moment-angle complex of two points.
Here $N = 4$, and the complement $M = \C^4\setminus V(\Phi_K)$ of this
quadric is diffeomorphic to $F_K\times\C^*\simeq S^3\times S^1$.
\end{example}

\begin{remark}
\label{rem:FdB-DS07}
The original application of Corollary~\ref{cor:FdB-ZK} is due to
Fern\'andez de Bobadilla~\cite{FdB}, building on Denham--Suciu~\cite{DS07}.
Taking $K$ to be one of the five obstruction graphs of \cite{DS07},
one obtains a weighted-homogeneous polynomial whose Milnor fiber
$F_K \simeq \ZK$ is simply connected and non-formal: the
non-formality is witnessed by a non-trivial triple Massey product
in $H^8(F_K;\Z)$, on classes in $H^3(F_K;\Z)\cong\Z^5$, 
with trivial indeterminacy \cite[\S 6.2]{DS07}.

This is consistent with the Dimca--Papadima criterion, even in the
generality of Proposition~\ref{prop:DP}, and the reason is
instructive.  That criterion transports $q$-formality from a quotient
$X = Y/G$ to $Y$ whenever $G$ acts trivially on $H^{\le q}(Y;\Q)$, so
its entire force lies in the $q$-formality of the base $X$.  For a
hyperplane arrangement that base is the projectivized complement
$U(\A) = F(\A)/\Z_n$, formal by Brieskorn's theorem, and trivial
monodromy then forces the formality of $F(\A)$
(Corollary~\ref{cor:DP}).  The polynomial $\Phi_K$ is built instead
from a coordinate subspace arrangement; its geometric monodromy is
again trivial (Theorem~\ref{thm:FdB}), so $\Z_n$ acts trivially on
$H^*(F_K;\Q)$, but the relevant base---the monodromy quotient
$F_K/\Z_n$---is not $q$-formal, as it must fail to be once
$F_K\simeq\ZK$ is non-formal.  The formality of the base, automatic
for arrangements, is exactly what is missing here.
\end{remark}

\subsection{Hochster's formula and connectivity}
\label{subsec:Hochster-conn}

Let $K_J$ denote the full subcomplex of $K$ induced on
$J\subseteq[m]$. A theorem of Hochster~\cite{Ho77}, in the form
of Buchstaber--Panov \cite{BP02} and Baskakov \cite{Ba02}, gives
a ring isomorphism
\begin{equation}
\label{eq:hochster}
H^*(\ZK;\Z) \cong \bigoplus_{J\subseteq[m]} \widetilde{H}^*(K_J;\Z),
\end{equation}
where a class in $\widetilde{H}^{p}(K_J;\Z)$ sits in
$H^{p+|J|+1}(\ZK;\Z)$.  The ring structure is given by
Baskakov's formula \cite{Ba02} (see also \cite{BBP04,Pa05}): 
the product of $\alpha\in\widetilde{H}^*(K_I;\Z)$ and
$\beta\in\widetilde{H}^*(K_J;\Z)$ is zero if $I\cap J\ne\emptyset$,
and is induced by $K_I * K_J \leftarrow K_{I\sqcup J}$ if
$I\cap J=\emptyset$.

The moment-angle complex $\ZK$ is $2$-connected for any simplicial
complex $K$ on $m \ge 2$ vertices \cite[Thm.~6.33(a)]{BP02}. 
Higher connectivity of $\ZK$ is controlled by the invariant 
\begin{equation}
\label{eq:nuk}
\nu(K) = \min\{\,|\sigma| : \sigma \text{ is a minimal non-face of } K\,\},
\end{equation}
the smallest cardinality of a minimal non-face of $K$. The connectivity 
bound below follows from the Ziegler--\v{Z}ivaljevi\'c theorem~\cite{ZZ93}, 
which gives the homotopy type of a linear subspace arrangement link as a 
wedge of joins (Theorem 2.4 there), together with the identification 
$M(\A_K)\simeq\ZK$ of~\cite{BP15}.

\begin{theorem}[\cite{ZZ93,BP15}]
\label{thm:ZZ}
Let $K$ be a simplicial complex on $[m]$ in which every minimal
non-face has cardinality at least $\nu+1$ (equivalently,
$\nu(K)\ge\nu+1$).  Then:
\begin{enumerate}[itemsep=1pt]
\item \label{zz1}
The coordinate subspace arrangement complement
$M(\A_K)$ is $2\nu$-connected.
\item \label{zz2}
The moment-angle complex $\ZK\simeq M(\A_K)$
is $2\nu$-connected.
\end{enumerate}
\end{theorem}

Equivalently: if $\nu(K)\ge\nu+1$, then $\ZK$ is $2\nu$-connected;
in particular, $\nu(K)\ge 2$ recovers the $2$-connectivity above, and
$\nu(K)\ge 3$ gives $4$-connectivity.

The factor of $2$ reflects the real codimension $2\nu(K)$ of the
smallest subspaces $L_\sigma$ (those with $|\sigma|=\nu(K)$): a complex
subspace arrangement whose subspaces all have complex codimension
$\ge c$ has a $(2c-2)$-connected complement.  This is sharp here: a
minimal non-face $\sigma$ of size $\nu(K)$ gives
$K_\sigma = \partial\Delta^{\nu(K)-1}\simeq S^{\nu(K)-2}$, contributing
through~\eqref{eq:hochster} a class in $H^{2\nu(K)-1}(\ZK;\Z)$, so
$\ZK$ is not $(2\nu(K)-1)$-connected.  More generally, $\nu(K)\ge c$
forces every full subcomplex $K_J$ to have a complete
$(c-2)$-skeleton, hence to be $(c-3)$-connected, so no class of
$\widetilde{H}^*(\ZK)$ occurs below degree $2\nu(K)-1$.

\subsection{Massey products in moment-angle complexes}
\label{subsec:massey-zk}
Throughout, $\k$ denotes a field. Baskakov \cite{Ba03} (see also Panov \cite{Pa05}) 
and Grbi\'c--Linton \cite{GL} construct non-trivial higher Massey products in the 
cohomology ring $H^*(\ZK;\k)$; the constructions are combinatorial and valid 
over any field or $\Z$ \cite[\S2]{GL}, taking place in the cellular cochain algebra
$C^*_{\mathrm{CW}}(\ZK;\k)$, which is a DGA, not a CDGA. Formality, by 
contrast, is a rational invariant (\S\ref{subsec:formality}): a property of the 
CDGA $A_{\mathrm{PL}}(\ZK;\Q)$.

This cellular DGA is not known to be weakly equivalent to
$A_{\mathrm{PL}}(\ZK;\Q)$, so a non-trivial Massey product in it does not by
itself force non-formality of the space.  That such products nonetheless
obstruct the rational formality of $\ZK$ was established by Denham--Suciu
\cite{DS07}, via Watkiss's theorem \cite[Cor.~10.10]{FHT}.  Since the
Grbi\'c--Linton products exist over $\Q$, the complexes constructed here are not
formal.

For their formality \emph{degrees} (Theorems~\ref{thm:main} and~\ref{thm:family})
we use the F\'elix--Tanr\'e CDGA model $A(K)\simeq A_{\mathrm{PL}}(\ZK;\Q)$ of
Theorem~\ref{thm:FT} \cite{FT09}.

A systematic study of \emph{lowest-degree} triple Massey products
in $H^*(\ZK;\Z)$ was carried out by Denham--Suciu~\cite{DS07}.
In degree $3$, i.e., on classes 
\[
\alpha_1,\alpha_2,\alpha_3\in
H^3(\ZK;\Z) = \bigoplus_{|J|=2}\widetilde{H}^0(K_J;\Z),
\]
the Massey product $\langle\alpha_1,\alpha_2,\alpha_3\rangle$ is
non-trivial (with trivial indeterminacy) if and only if the $1$-skeleton 
of $K$ contains an induced subgraph isomorphic to one of five specific
\emph{obstruction graphs} \cite[Thm.~6.1.1]{DS07}.  The smallest
such $K$ lives on $6$ vertices, and its moment-angle complex is
simply connected. 
The classification was subsequently completed by Grbi\'c--Linton
\cite{GL20}: there is a \emph{sixth} obstruction graph (on six vertices),
which yields a non-trivial triple Massey product with \emph{non-trivial}
indeterminacy; the DS classification implicitly assumed trivial
indeterminacy (all five products there are decomposable,
hence have trivial indeterminacy).
The $\Phi_K$-polynomial associated to the sixth graph gives a
non-formal Milnor fiber of the same type as the DS examples
($2$-connected, triple Massey product in $H^8$), but with
non-trivial indeterminacy.

Crucially for this paper, the Denham--Suciu analysis is confined
to $K^i = S^0$ (two disjoint vertices), giving $p_i = 0$, and to
$n=3$. The five obstruction graphs have edges as their minimal non-faces,
so $\nu(K) = 2$, and Theorem~\ref{thm:ZZ} gives that $\ZK$ is
$2$-connected. Producing Milnor fibers with
higher connectivity requires non-trivial Massey products in
moment-angle complexes with larger minimal non-faces, i.e., with
$\nu(K) \ge k+1$ for $k\ge 2$---and for $n$-fold products with
$n\ge 3$ arbitrary. This is precisely what the Grbi\'c--Linton
construction achieves.

\subsection{The Grbi\'c--Linton constructions}
\label{subsec:GL}

We recall the two constructions of Grbi\'c--Linton \cite{GL}: joins with star
deletions (Construction~\ref{cons:GL1}), which feeds the proof of
Theorem~\ref{thm:main}, and edge stretching (Construction~\ref{cons:GL2}); the
corresponding non-triviality statements are Theorems~\ref{thm:GL-joins}
and~\ref{thm:GL-edge}. We begin with a basic operation on simplicial complexes.

\medskip\noindent\textbf{Star deletion.}
For a simplicial complex $K$ and a simplex $I\in K$, the
\emph{star deletion} of $K$ at $I$ is
\begin{equation}
\label{eq:stars-del}
\sd_I K = \{J\in K : I\not\subseteq J\}.
\end{equation}
Star deletions at non-nested simplices commute \cite[Lem.~3.2]{GL}.

\begin{cons}[joins and star deletions; {\cite[Constr.~3.5]{GL}}]
\label{cons:GL1}
Let $K^1,\dots,K^n$ be simplicial complexes on pairwise disjoint vertex sets,
each not a simplex.  From this data Grbi\'c--Linton build a simplicial complex
$K$, obtained from the join $K^1*\cdots*K^n$ by a prescribed sequence of star
deletions, whose moment-angle complex $\ZK$ carries a non-trivial $n$-fold
Massey product.  The construction is as follows.

For each $i$, fix a non-trivial class $\alpha_i\in\widetilde{H}^{p_i}(K^i_{J_i};\k)$
for some $J_i\subseteq[m_i]$, and let $a_i = \sum_{\sigma\in S_{a_i}}
c_\sigma\chi_\sigma$ be a cocycle representative, supported on the set $S_{a_i}$
of $p_i$-simplices.  For each $\sigma\in S_{a_i}$ choose a distinguished vertex
$v_\sigma\in\sigma$ and set
\[
  P_\sigma = \bigl\{\sigma'\in K^i : \sigma'\ \text{a}\ p_i\text{-simplex},\
  \sigma\cap\sigma' = \sigma\setminus v_\sigma \bigr\},
\]
the $p_i$-simplices meeting $\sigma$ in exactly the facet $\sigma\setminus v_\sigma$
(so $v_\sigma\notin\sigma'$ and $\sigma\notin P_\sigma$).  The set $P_{a_i}$ is
then assembled by the ordered procedure of \cite[Constr.~3.5]{GL}: fix an order
on $S_{a_i}$ and choose the subsequence $\sigma^{(1)},\dots,\sigma^{(\ell)}$ with
$\sigma^{(1)}$ the first element of $S_{a_i}$ and $\sigma^{(t)}$ the first element
of $S_{a_i}\setminus\bigl(P_{\sigma^{(1)}}\cup\cdots\cup P_{\sigma^{(t-1)}}\bigr)$
after $\sigma^{(t-1)}$; then
$P_{a_i} = P_{\sigma^{(1)}}\cup\cdots\cup P_{\sigma^{(\ell)}}$.  When $|S_{a_i}|=1$
(as for $K^i=\partial\Delta^q$ below) this reduces to $P_{a_i}=P_\sigma$.

The complex $K$ is then
\[
  K = \biggl(\prod_{\substack{1\le i<j\le n\\ (i,j)\ne(1,n)}}
  \sd_{\sigma\cup\sigma'}\biggr)\,  \bigl(K^1*\cdots*K^n\bigr),
\]
the star deletions running over $\sigma\in S_{a_i}$ and $\sigma'\in P_{a_j}$. 
Figure~\ref{fig:star-deletion} illustrates the star-deletion step for $K^1=K^2=S^0$.
\end{cons}

\begin{figure}[ht]
\centering
\begin{tikzpicture}[scale=1.25, inner sep=2pt,
  vx/.style={circle, fill=black, minimum size=5pt},
  lbl/.style={font=\small}]

\begin{scope}[xshift=0cm]
  \node[vx] (a) at (0,1.5) {};
  \node[lbl] at (-0.25,1.5) {$a$};
  \node[vx] (b) at (0,0) {};
  \node[lbl] at (-0.25,0) {$b$};
  \node[lbl] at (0,-0.7) {$K^1 = S^0$};

  \node[vx] (c) at (1.2,1.5) {};
  \node[lbl] at (1.45,1.5) {$c$};
  \node[vx] (d) at (1.2,0) {};
  \node[lbl] at (1.45,0) {$d$};
  \node[lbl] at (1.2,-0.7) {$K^2 = S^0$};

  \node[lbl, blue] at (-0.35,0.75) {$\alpha_1$};
  \node[lbl, red]  at (1.55,0.75)  {$\alpha_2$};
\end{scope}

\draw[->, thick] (2.05,0.75) -- (2.85,0.75)
  node[midway, above, font=\small]{join};

\begin{scope}[xshift=3.4cm]
  \coordinate (a) at (0,1.5);
  \coordinate (b) at (0,0);
  \coordinate (c) at (1.4,1.5);
  \coordinate (d) at (1.4,0);

  \draw[thick] (a) -- (d);
  \draw[thick] (b) -- (c);
  \draw[thick] (b) -- (d);
  \draw[thick, red] (a) -- (c);

  \node[vx] at (a) {}; \node[lbl] at (-0.25,1.5) {$a$};
  \node[vx] at (b) {}; \node[lbl] at (-0.25,0)   {$b$};
  \node[vx] at (c) {}; \node[lbl] at (1.65,1.5)  {$c$};
  \node[vx] at (d) {}; \node[lbl] at (1.65,0)    {$d$};

  \node[lbl, red, font=\scriptsize] at (0.7,1.72) {$\{a,c\}$};
  \node[lbl] at (0.7,-0.7) {$K^1 * K^2$};
\end{scope}

\draw[->, thick] (5.5,0.75) -- (6.3,0.75)
  node[midway, above, font=\small]{$\sd_{\{a,c\}}$};

\begin{scope}[xshift=6.9cm]
  \coordinate (a) at (0,1.5);
  \coordinate (b) at (0,0);
  \coordinate (c) at (1.4,1.5);
  \coordinate (d) at (1.4,0);

  \draw[thick] (a) -- (d);
  \draw[thick] (b) -- (c);
  \draw[thick] (b) -- (d);
  \draw[dashed, red, thick] (a) -- (c);

  \node[vx] at (a) {}; \node[lbl] at (-0.25,1.5) {$a$};
  \node[vx] at (b) {}; \node[lbl] at (-0.25,0)   {$b$};
  \node[vx] at (c) {}; \node[lbl] at (1.65,1.5)  {$c$};
  \node[vx] at (d) {}; \node[lbl] at (1.65,0)    {$d$};

  \node[lbl] at (0.7,-0.7) {$\sd_{\{a,c\}}(K^1*K^2)$};
\end{scope}

\end{tikzpicture}
\caption{The GL star-deletion step for $K^1 = K^2 = S^0$
(pairs of disjoint vertices carrying classes $\alpha_1$, $\alpha_2$),
illustrating one building block of Construction~\ref{cons:GL1}.
\emph{Left}: the two complexes separately.
\emph{Middle}: their join $K^1*K^2 = K_{2,2}$, the complete bipartite
graph on $\{a,b\}\sqcup\{c,d\}$; the edge $\{a,c\}$ (red) is selected
for star deletion.
\emph{Right}: after $\sd_{\{a,c\}}$, the edge $\{a,c\}$ is removed
and becomes a new minimal non-face (dashed red).  In the full $n=3$
construction, a third factor $K^3=S^0$ is joined and a second round
of star deletions trivializes the remaining sub-product
$\langle\alpha_2,\alpha_3\rangle$.}
\label{fig:star-deletion}
\end{figure}
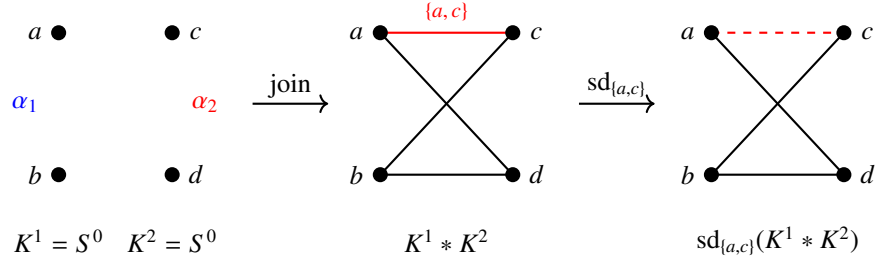

\begin{cons}[edge stretching; {\cite[\S4]{GL}}]
\label{cons:GL2}
An \emph{edge contraction} $\varphi\colon K\to\widehat K$ collapses a single
edge $\{u,v\}\in K$ to a vertex, and is a homotopy equivalence when it satisfies
the \emph{link condition} $\lk_K\{u\}\cap\lk_K\{v\}=\lk_K\{u,v\}$ \cite{DEGN}.
The inverse operation is \emph{edge stretching}: a vertex $w\in\widehat K$ is
split into two new vertices $u,v$ joined by a new edge $\{u,v\}$, with the
adjacency of $w$ distributed between $u$ and~$v$. 
Figure~\ref{fig:edge-stretch} illustrates the operation for the $4$-cycle $C_4$.
\end{cons}

\begin{figure}[ht]
\centering
\begin{tikzpicture}[scale=1.25, inner sep=2pt,
  vx/.style={circle, fill=black, minimum size=5pt},
  lbl/.style={font=\small}]

\begin{scope}[xshift=-0.1cm]
  \coordinate (1) at (0,1.4);
  \coordinate (2) at (1.4,1.4);
  \coordinate (3) at (1.4,0);
  \coordinate (4) at (0,0);

  \draw[thick] (2) -- (3);
  \draw[thick] (3) -- (4);
  \draw[thick, red] (4) -- (1);
  \draw[thick, red] (1) -- (2);

  \node[vx] at (1) {}; \node[lbl, red] at (-0.25,1.4)  {$1$};
  \node[vx] at (2) {}; \node[lbl]      at (1.65,1.4)   {$2$};
  \node[vx] at (3) {}; \node[lbl]      at (1.65,0)     {$3$};
  \node[vx] at (4) {}; \node[lbl]      at (-0.25,0)    {$4$};

  \node[lbl] at (0.7,-0.65) {$\widehat{K} = C_4$};
\end{scope}

\draw[->, thick] (2.1,1.05) -- (3.4,1.05)
  node[midway, above, yshift=2pt, font=\small]{stretch $1$};

\begin{scope}[xshift=4.2cm]
  \coordinate (1) at (0,1.4);
  \coordinate (5) at (1.0,1.8);
  \coordinate (2) at (1.75,0.9);
  \coordinate (3) at (1.4,0);
  \coordinate (4) at (0,0);

  \draw[thick] (2) -- (3);
  \draw[thick] (3) -- (4);
  \draw[thick] (4) -- (1);
  \draw[thick, red] (1) -- (5);
  \draw[thick, red] (5) -- (2);

  \node[vx] at (1) {}; \node[lbl]      at (-0.25,1.4)  {$1$};
  \node[vx] at (5) {}; \node[lbl, red] at (1.0,2.1)    {$5$};
  \node[vx] at (2) {}; \node[lbl]      at (2.0,0.9)    {$2$};
  \node[vx] at (3) {}; \node[lbl]      at (1.65,0)     {$3$};
  \node[vx] at (4) {}; \node[lbl]      at (-0.25,0)    {$4$};

  \node[lbl] at (0.85,-0.65) {$K = C_5$};
\end{scope}

\draw[->, thick] (3.65,0.35) -- (1.85,0.35)
  node[midway, below, yshift=-1pt, font=\small]{contract $\{1,5\}$};

\end{tikzpicture}
\caption{Edge stretching (top arrow) splits vertex $1$ of
$\widehat{K}=C_4$ into a new edge $\{1,5\}$ (red), producing
$K=C_5$; the edges incident to $1$ in $\widehat{K}$ are
redistributed as $\{4,1\}$ and $\{5,2\}$ in $K$.
The inverse (bottom arrow) is the edge contraction
$\varphi\colon K\to\widehat{K}$, collapsing $\{1,5\}\mapsto 1$;
it satisfies the link condition since
$\lk_K\{1\}\cap\lk_K\{5\} = \{4\}\cap\{2\} = \emptyset
= \lk_K\{1,5\}$.}
\label{fig:edge-stretch}
\end{figure}
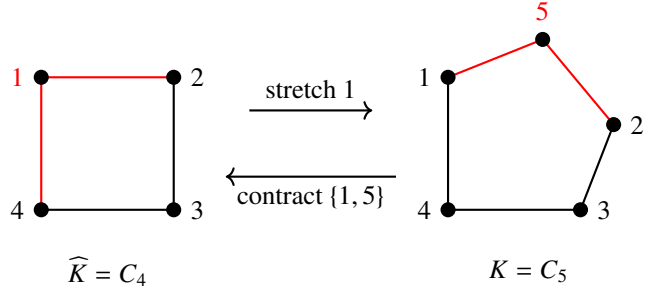

\begin{theorem}[{\cite[Thm.~3.17]{GL}}]
\label{thm:GL-joins}
Let $K$ be the simplicial complex of Construction~\ref{cons:GL1}.
Then the $n$-fold Massey product
$\langle\alpha_1,\dots,\alpha_n\rangle
\subset H^{p_1+\cdots+p_n+|J_1\cup\cdots\cup J_n|+2}(\ZK;\Q)$
is non-trivial.  In particular, $\ZK$ is non-formal.
\end{theorem}

This star-deletion step trivializes $\langle\alpha_1,\alpha_2\rangle$
as a sub-product in the $n=3$ construction (cf.\ Figure~\ref{fig:star-deletion}). 
Each star deletion $\sd_{\sigma_i\cup\sigma_j'}$ removes a simplex
from the join and introduces a new minimal non-face; by Hochster's
formula~\eqref{eq:hochster}, this new missing face contributes a
class in $H^*(\ZK;\Z)$ whose cup product with the relevant
$\alpha_i$ vanishes, making the sub-product
$\langle\alpha_i,\dots,\alpha_j\rangle$ trivial.
The star deletions are arranged so that all proper sub-products
$\langle\alpha_i,\dots,\alpha_j\rangle$ for $(i,j)\ne(1,n)$ are
trivialized in this way, ensuring the full $n$-fold product is
defined; non-triviality is then established by an explicit
evaluation on a cycle.

\begin{remark}
\label{rem:GL-connectivity}
When $K^i = \partial\Delta^{k+1}$ (the boundary of the
$(k+1)$-simplex, so $p_i = k$), every star-deleted simplex
$\sigma_i\cup\sigma_j'$ has cardinality $2k+2$, while the surviving
minimal non-faces $[k+2]_i$ have cardinality $k+2$; thus
$\nu(K) = k+2$.  By Theorem~\ref{thm:ZZ}, $\ZK$ is therefore
$(2k+2)$-connected---consistent with the value $2q=2k+2$ obtained for
the $n=3$ family in Theorem~\ref{thm:family}.
\end{remark}

\begin{theorem}[{\cite[Thm.~4.12]{GL}}]
\label{thm:GL-edge}
Let $\widehat{K}$ carry a non-trivial $n$-fold Massey product in
$H^*(\mathcal{Z}_{\widehat{K}};\Z)$.  If $K$ maps onto
$\widehat{K}$ by a sequence of link-condition edge contractions,
then $H^*(\ZK;\Z)$ also contains a non-trivial $n$-fold Massey product.
\end{theorem}

Since a link-condition edge contraction is a homotopy equivalence,
$\varphi^*\colon H^*(\mathcal{Z}_{\widehat{K}};\Z)\to H^*(\ZK;\Z)$
is an isomorphism.  The non-trivial Massey product in
$H^*(\mathcal{Z}_{\widehat{K}};\Z)$ pulls back to one that remains
non-trivial in $H^*(\ZK;\Z)$: stretching edges inflates
the moment-angle complex without destroying the non-formality
obstruction.  Iterating this construction produces infinitely many
non-homeomorphic moment-angle complexes carrying non-trivial
Massey products in arbitrarily large cohomological degrees.

\section{Main result}
\label{sec:main}

The following theorem is the main result of this paper.  Its proof
combines the Fern\'{a}ndez de Bobadilla realization
(Theorem~\ref{thm:FdB}), the Grbi\'{c}--Linton join construction
(Theorem~\ref{thm:GL-joins}), and the Ziegler--\v{Z}ivaljevi\'{c}
connectivity estimate (Theorem~\ref{thm:ZZ}).

\begin{theorem}
\label{thm:main}
For every $k \ge 1$ and $n \ge 3$, there exists a weighted-homogeneous
polynomial $f\colon \C^N \to \C$ (for some $N$ depending on $k$ and
$n$) with trivial geometric monodromy, such that the Milnor fiber
$F = f^{-1}(1)$ is $(2k+2)$-connected and non-formal, with
non-formality witnessed by a non-trivial $n$-fold Massey product in
$H^{2n(k+1)+2}(F;\Q)$.
\end{theorem}

\begin{proof}
We construct the desired polynomial in four steps.

\smallskip
\noindent\textbf{Step 1: Construction of the simplicial complex $K$.}
For each $i = 1, \dots, n$, let $K^i = \partial\Delta^{k+1}$ be the
boundary of the $(k+1)$-simplex, on vertex set $[k+2]_i$ (disjoint
copies for different $i$).  This is a simplicial complex on $m_i = k+2$
vertices whose unique minimal non-face is the full simplex
$[k+2]_i$ of cardinality $k+2$.  In particular, $K^i$ is not a
simplex, so there exists a non-trivial class $\alpha_i \in
\widetilde{H}^k(K^i;\Z) \cong \Z$.

Apply the Grbi\'{c}--Linton Construction~\ref{cons:GL1} to the
join $K^1 * \cdots * K^n$: choose cocycle representatives $a_i$ for
$\alpha_i$, form the sets $S_{a_i}$ (supports of $a_i$) and $P_{a_i}$
(the associated partner simplices), and let
\[
K = \Bigl(\prod_{\substack{1\le i<j\le n\\ (i,j)\ne(1,n)}}
  \sd_{\sigma_i\cup\sigma_j'}\Bigr) (K^1 * \cdots * K^n),
\]
where the star deletions run over all $\sigma_i \in S_{a_i}$,
$\sigma_j' \in P_{a_j}$, $1 \le i < j \le n$, $(i,j) \ne (1,n)$.
The resulting simplicial complex $K$ lies on $m = n(k+2)$ vertices.
(See Figure~\ref{fig:star-deletion} for the case $K^1=K^2=S^0$, $n=3$,
which corresponds to $k=0$.)

\smallskip
\noindent\textbf{Step 2: Connectivity of $\ZK$.}
We show that every minimal non-face of $K$ has cardinality $\ge k+2$.
The minimal non-faces of the join $K^1 * \cdots * K^n$ are exactly the
missing faces $[k+2]_i$, of cardinality $k+2$.  A star deletion
$\sd_{\sigma_i \cup \sigma_j'}$ removes precisely the faces containing
$\sigma_i \cup \sigma_j'$, hence makes $\sigma_i \cup \sigma_j'$ a
minimal non-face and creates no other: since $\sigma_i \in S_{a_i}$
and $\sigma_j' \in P_{a_j}$ are $k$-simplices on disjoint vertex sets,
$\sigma_i \cup \sigma_j'$ has cardinality $2k+2$, and each of its
proper subsets---of cardinality $\le 2k+1$, too small to contain any
deleted face---is a face of the join that survives every deletion.
The minimal non-faces of $K$ are therefore the $[k+2]_i$ (cardinality
$k+2$) and the deleted faces $\sigma_i \cup \sigma_j'$ (cardinality
$2k+2$); thus $\nu(K) = k+2$, attained by the $[k+2]_i$.

By Theorem~\ref{thm:ZZ}, $\nu(K)\ge k+2 = (k+1)+1$ implies that
$\ZK$ is $(2k+2)$-connected.  Hochster's formula~\eqref{eq:hochster}
moreover locates the first non-vanishing reduced cohomology in degree
$2\nu(K)-1 = 2k+3$: a nonzero summand $\widetilde{H}^{p}(K_J;\Z)$
requires $J$ to contain a minimal non-face, and the smallest such $J$
are the single faces $[k+2]_i$, for which
$K_J = \partial\Delta^{k+1}\simeq S^k$ contributes $\widetilde{H}^k$
at total degree $(k+2)+k+1 = 2k+3$.  Hence $H^{2k+2}(\ZK;\Z)=0$ while
$H^{2k+3}(\ZK;\Z)\ne 0$; the classes $\alpha_i$ used below live in
this bottom degree.

\smallskip
\noindent\textbf{Step 3: Non-formality of $\ZK$ via a non-trivial
$n$-fold Massey product.}
By Theorem~\ref{thm:GL-joins}, the construction of Step~1 produces a non-trivial
$n$-fold Massey product
\[
\langle \alpha_1, \dots, \alpha_n \rangle \subset H^{2n(k+1)+2}(\ZK;\Q),
\]
where $\alpha_i \in H^{2k+3}(\ZK;\Q)$ is the Hochster class of the generator of
$\widetilde{H}^k(K^i;\Q)$.  The degree $2n(k+1)+2$ follows from the GL degree
formula $p_1+\cdots+p_n+|J_1\cup\cdots\cup J_n|+2$ with $p_i=k$ and
$|J_1\cup\cdots\cup J_n|=n(k+2)$.  Hence $\ZK$ is non-formal
(\S\ref{subsec:massey-zk}).

\smallskip
\noindent\textbf{Step 4: Realization as a Milnor fiber.}
Apply Corollary~\ref{cor:FdB-ZK} to the simplicial complex $K$
constructed in Step~1: with $F_1,\dots,F_r$ the facets of $K$ on
$m=n(k+2)$ vertices, the weighted-homogeneous polynomial
\begin{equation}
\label{eq:Phi-thm-main}
\Phi_K(x_1,\dots,x_m,\,y_1,\dots,y_r) = 
\sum_{j=1}^{r} y_j \prod_{i\in[m]\setminus F_j} x_i
\end{equation}
has trivial geometric monodromy, and its Milnor fiber 
$F_K = \Phi_K^{-1}(1)$ is homotopy equivalent to $\ZK$.

\smallskip
\noindent\textbf{Conclusion.}
Set $f = \Phi_{K}$ and $N = m + r = n(k+2) + r$.  By Step~4 the Milnor
fiber $F = F_K$ is homotopy equivalent to $\ZK$ and has trivial
geometric monodromy.  By Step~2, $\ZK$ is $(2k+2)$-connected, so $F$ is
$(2k+2)$-connected.  By Step~3, $\ZK$ is non-formal with a non-trivial
$n$-fold Massey product in $H^{2n(k+1)+2}(\ZK;\Z)$, which pulls back to
a non-trivial $n$-fold Massey product in $H^{2n(k+1)+2}(F;\Z)$.  This
completes the proof.
\end{proof}

An explicit infinite family realizing Theorem~\ref{thm:main} for $n=3$ 
and $K^i = \partial\Delta^{k+1}$ is given in Theorem~\ref{thm:family}, 
where the exact formality degree $6(k+1)$ is also determined.

\begin{remark}
\label{rem:fdb-massey}
The case $k=0$, $n=3$ (outside the range $k\ge 1$ of
Theorem~\ref{thm:main}) recovers the Fern\'{a}ndez de Bobadilla
example \cite{FdB}: here $K^i = \partial\Delta^1 = S^0$ consists
of two disjoint vertices, and the star-deletion construction
of Step~1 produces one of the five obstruction graphs classified
in \cite[\S 6.1]{DS07}.  The resulting moment-angle complex $\ZK$
is $2$-connected, and the non-trivial triple Massey product lives 
in $H^{2\cdot3\cdot1+2}(\ZK;\Z)=H^8(\ZK;\Z)$,
witnessing non-formality as in \cite[\S 6.2]{DS07}.
\end{remark}


\section{Explicit families and the Kato--Matsumoto bound}
\label{sec:examples}

This section develops the construction of Theorem~\ref{thm:main}
in explicit detail. \S\ref{subsec:sing} computes the dimension of the 
singular set $\Sing(V(\Phi_K))$ for any polynomial 
$\Phi_K=\sum_{F} y_F g_F$ of the Fern\'andez de Bobadilla form, 
via a combinatorial formula for the rank of the Jacobian of the 
generators (Lemma~\ref{lem:jac-rank}); it specializes 
to $s=m+r-2\nu$ under a mild non-degeneracy condition. 
\S\ref{subsec:fdb-example} works out the original
Fern\'andez de Bobadilla polynomial as an example.
\S\S\ref{subsec:Kq-combinatorics} and \ref{subsec:Phi-Kq-Milnor}
focus on the one-parameter family arising from the Grbi\'c--Linton
Construction~\ref{cons:GL1} with $n=3$ and $K^i = \partial\Delta^q$:
\S\ref{subsec:Kq-combinatorics} introduces the simplicial
complexes $K(q)$ and establishes their combinatorial properties
(facet structure, homotopy type, and the cohomology of $\ZK$);
\S\ref{subsec:Phi-Kq-Milnor} then states the main theorem about the
polynomial $\Phi_K$ and its Milnor fiber, and grounds it in the
first concrete example $K(2)$.
\S\ref{subsec:further-remarks} records the closed-form invariants of
the family, its general-$n$ analogue, and the role of
weighted-homogeneity, with Macaulay2 \cite{M2} confirming the explicit ranks.
Throughout, the Kato--Matsumoto (KM) bound $N-s-2$ on the
connectivity of $F_K\simeq\ZK$ is compared with the true
connectivity.


\subsection{The dimension of the singular set}
\label{subsec:sing}

We compute $s=\dim\Sing(V(\Phi_K))$ for an arbitrary polynomial
$\Phi_K=\sum_F y_F\,g_F$ of the Fern\'andez de Bobadilla form.  The
computation rests on a stratum-by-stratum evaluation of the rank of the
Jacobian $(\partial g_F/\partial x_i)$, which is purely combinatorial
(Lemma~\ref{lem:jac-rank}); the resulting dimension formula
(Proposition~\ref{prop:sing}) specializes to $s=m+r-2\nu$ under a mild
non-degeneracy condition satisfied by every complex considered here
(Remark~\ref{rem:sing-hypothesis}).

For $Z\subseteq[m]$ write
\[
L_Z^\circ = \{x\in\C^m : x_i = 0 \iff i\in Z\} \cong (\C^*)^{m-|Z|};
\] 
these are the strata of the coordinate stratification of $\C^m$. Their union over the
non-faces is exactly the arrangement,
$V(\A_K) = \bigsqcup_{Z\notin K} L_Z^\circ$, whose maximal (top-dimensional)
strata are the $L_\sigma^\circ$ with $\sigma$ a minimal non-face of minimum
cardinality $\nu$. The Jacobian $J = (\partial g_F/\partial x_j)$ has constant
rank on each stratum $L_Z^\circ$.

\begin{lemma}
\label{lem:jac-rank}
For every $Z\subseteq[m]$ containing a minimal non-face,
\[
  \operatorname{rank} J|_{L_Z^\circ}
  =\#\{\,i\in Z: Z\setminus\{i\}\in K\,\}
  =\Bigl|\,\textstyle\bigcap_{\tau\subseteq Z}\tau\,\Bigr|
  \eqqcolon \rho_Z,
\]
the intersection running over the minimal non-faces $\tau\subseteq Z$.
\end{lemma}

\begin{proof}
For $i\notin F$ one has $\partial g_F/\partial x_i=\prod_{l\notin F\cup\{i\}}x_l$,
while $\partial g_F/\partial x_i=0$ for $i\in F$.  Restricted to $L_Z^\circ$,
the $(i,F)$ entry is therefore nonzero iff $i\notin F$ and
$Z\setminus\{i\}\subseteq F$, i.e.\ iff $Z\setminus F=\{i\}$ (which forces
$i\in Z$).  For each facet $F$ this determines at most one row index, so
$J|_{L_Z^\circ}$ has at most one nonzero entry per column; its rank thus
equals the number of nonzero rows, namely
$\#\{i\in Z:\exists\,F,\ Z\setminus F=\{i\}\}=\#\{i\in Z:Z\setminus\{i\}\in K\}$,
the second equality because $Z\notin K$: if $Z\setminus\{i\}\in K$ then,
$Z$ not being a face, some facet contains $Z\setminus\{i\}$ but not $i$.
Finally $Z\setminus\{i\}\in K$ iff $Z\setminus\{i\}$ contains no minimal
non-face, i.e.\ iff $i$ lies in every minimal non-face $\tau\subseteq Z$;
the count therefore equals $\bigl|\bigcap_{\tau\subseteq Z}\tau\bigr|$.
\end{proof}

\begin{prop}
\label{prop:sing}
Let $\nu=\nu(K)$.  Then
\begin{equation}
\label{eq:dimSing-general}
  \dim\Sing(V(\Phi_K)) = m+r-\min_Z\bigl(|Z|+\rho_Z\bigr),
\end{equation}
the minimum over all $Z\subseteq[m]$ containing a minimal non-face, and
$V(\A_K)\times\{0\}\subseteq\Sing(V(\Phi_K))$.  Moreover
$\min_Z(|Z|+\rho_Z)\le 2\nu$; if $|Z|+\rho_Z\ge 2\nu$ for every such $Z$,
then
\begin{equation}
\label{eq:dimSing}
  \dim\Sing(V(\Phi_K)) = m+r-2\nu,
  \qquad
  \codim_{V(\Phi_K)}\Sing = 2\nu-1.
\end{equation}
\end{prop}

\begin{proof}
Since $\Phi_K$ is weighted-homogeneous, $\Sing(V(\Phi_K))$ is the critical
locus $\{\nabla\Phi_K=0\}$.  The equations
$\partial\Phi_K/\partial y_F=g_F(x)$ vanish exactly when $x\in V(\A_K)$, as
the $g_F$ generate $I_{K^\vee}=I(V(\A_K))$ by~\eqref{eq:IVAK}; for such $x$
the equations $\partial\Phi_K/\partial x_i=\sum_F y_F\,\partial g_F/\partial x_i$
read $J(x)\,y=0$.  Hence $\Sing(V(\Phi_K))$ is the total space of the kernel
bundle of $J$ over $V(\A_K)$, and $V(\A_K)\times\{0\}$ lies in it.  Over each
stratum $L_Z^\circ\cong(\C^*)^{m-|Z|}$ this bundle has rank
$r-\rho_Z$ by Lemma~\ref{lem:jac-rank}, hence dimension $(m-|Z|)+(r-\rho_Z)$;
maximizing over strata gives~\eqref{eq:dimSing-general}.

For $Z=\sigma$ a minimal non-face of minimum cardinality $\nu$, $\sigma$ is
the only minimal non-face it contains, so $\rho_\sigma=|\sigma|=\nu$ and
$|Z|+\rho_Z=2\nu$; thus $\min_Z(|Z|+\rho_Z)\le 2\nu$.  Under the stated
inequality this minimum equals $2\nu$, so $\dim\Sing(V(\Phi_K))=m+r-2\nu$;
as $\dim V(\Phi_K)=N-1=m+r-1$, the codimension is $2\nu-1$.
\end{proof}

\begin{remark}
\label{rem:sing-hypothesis}
The condition $|Z|+\rho_Z\ge 2\nu$ is combinatorial and mild.  On the
top stratum ($Z=\sigma$, $|\sigma|=\nu$) it holds automatically: as
$\sigma$ is a minimal non-face, each $\sigma\setminus\{i\}$ is a face, so
some facet contains it but not $i$, whence $\rho_\sigma=\nu$ and
$|Z|+\rho_Z=2\nu$.  For larger $Z$ it is automatic when $Z$
contains at most two minimal non-faces: with one, $\rho_Z=|\tau|$ and
$|Z|+\rho_Z\ge 2|\tau|\ge 2\nu$; with two, $\rho_Z=|\tau_1\cap\tau_2|$ while
$|Z|\ge|\tau_1\cup\tau_2|=|\tau_1|+|\tau_2|-|\tau_1\cap\tau_2|$, so again
$|Z|+\rho_Z\ge|\tau_1|+|\tau_2|\ge 2\nu$.  It can fail only for three or more
minimal non-faces with empty common intersection inside a set of size
$<2\nu$.  Both the family $K(q)$ of \S\ref{subsec:Kq-combinatorics} and the
Fern\'andez de Bobadilla complex of \S\ref{subsec:fdb-example} satisfy it,
so Proposition~\ref{prop:sing} yields $s=m+r-2\nu$ in each case; the minimum
$2\nu$ is attained on the top stratum and on certain deeper strata as well,
without being exceeded (for instance $Z=\{1,2,3\}$ in the FdB example, with
$\rho_Z=1$, and $Z=V_1\cup V_2$ for $K(q)$, with $\rho_Z=0$, both giving
$|Z|+\rho_Z=2\nu$).
\end{remark}

\subsection{The Fern\'andez de Bobadilla polynomial}
\label{subsec:fdb-example}

Let $K$ be the flag complex on $m = 6$ vertices $\{1, \dots, 6\}$ 
whose minimal non-faces are the five consecutive edges 
$\{1,2\}, \{2,3\}, \{3,4\}, \{4,5\}$, $\{5,6\}$ (the $1$-skeleton 
is the path graph $P_5$).  Hence $\nu(K) = 2$.  The facets of $K$ 
are the maximal independent sets of $P_5$: the four $3$-element 
sets $\{1,3,5\}$, $\{1,3,6\}$, $\{1,4,6\}, \{2,4,6\}$ and the 
single $2$-element set $\{2, 5\}$, giving $r = 5$, $N = 11$.  
Applying Corollary~\ref{cor:FdB-ZK}, the generators of $I_{K^\vee}$ 
are
\[
g_{\{1,3,5\}}=x_2x_4x_6,\;\;
g_{\{1,3,6\}}=x_2x_4x_5,\;\;
g_{\{1,4,6\}}=x_2x_3x_5,\;\;
g_{\{2,4,6\}}=x_1x_3x_5,\;\;
g_{\{2,5\}}=x_1x_3x_4x_6,
\]
and
\[
\Phi_K  = y_1\,x_1x_3x_4x_6 + y_2\,x_1x_3x_5 + y_3\,x_2x_3x_5
    + y_4\,x_2x_4x_6 + y_5\,x_2x_4x_5
\]
is weighted-homogeneous of degree $5$, with $w(x_j) = 1$, 
$w(y_1) = 1$, $w(y_2) = w(y_3) = w(y_4) = w(y_5) = 2$.  This is 
exactly the polynomial of \cite[\S 3]{FdB}.

The path $1\text{--}2\text{--}\cdots\text{--}6$ is one of the 
five Denham--Suciu obstruction graphs \cite[Thm.~6.1.1]{DS07}, so 
$H^*(\ZK; \Z)$ carries a non-trivial triple Massey product 
\[
\langle \alpha_{12}, \alpha_{34}, \alpha_{56} \rangle  \in H^{8}(\ZK; \Z)
\]
with trivial indeterminacy \cite[\S 6.2]{DS07}, where 
$\alpha_{ij} \in H^3(\ZK;\Z)$ is the Hochster class of the 
missing edge $\{i, j\}$.  Hence $F_K \simeq \ZK$ is $2$-connected 
and non-formal, with Poincar\'e polynomial 
$P_{F_K}(t) = 1 + 5t^3 + 4t^4 + 3t^6 + 4t^7 + t^8$ 
(see \cite[\S 3]{FdB}).

The non-degeneracy condition of Proposition~\ref{prop:sing} holds here.  
For a minimal edge $\sigma$, Lemma~\ref{lem:jac-rank} gives $\rho_\sigma = 2 = \nu$
(for instance $\sigma=\{1,2\}$: the facet $\{2,4,6\}$ contains $\{2\}$ but
not $1$, and $\{1,3,5\}$ contains $\{1\}$ but not $2$), so
$|\sigma|+\rho_\sigma = 2\nu$.  The condition $|Z|+\rho_Z\ge 2\nu$ holds for
every non-face $Z$: it is automatic when $Z$ contains at most two of the
edges (Remark~\ref{rem:sing-hypothesis}), while any three or more edges of
the path $P_5$ span at least $4=2\nu$ vertices, so $|Z|\ge 2\nu$.  Hence
Proposition~\ref{prop:sing} gives 
$s = \dim\Sing(V(\Phi_K)) = m + r - 2\nu = 6 + 5 - 4 = 7$, 
and the Kato--Matsumoto bound $N - s - 2 = 11 - 7 - 2 = 2$ 
equals the true connectivity of $F_K$: the bound is sharp, 
exactly as for the family of Theorem~\ref{thm:family}.


\subsection{Combinatorial properties of $K(q)$}
\label{subsec:Kq-combinatorics}

We now introduce the one-parameter family of simplicial complexes 
$K = K(q)$ obtained from the Grbi\'c--Linton Construction~\ref{cons:GL1} 
with $n = 3$ and $K^i = \partial\Delta^q$, and establish their 
combinatorial properties: facet structure (Lemma~\ref{lem:facets}), 
type-B facet count (Lemma~\ref{lem:facets}), homotopy type 
(Proposition~\ref{prop:K-homotopy}), and the cohomology 
of $\ZK$ (Corollary~\ref{cor:cohomology}).

Throughout this and the next subsection, fix an integer $q \ge 1$ and set
$V = V_1 \sqcup V_2 \sqcup V_3$, the disjoint union of three vertex sets each of
cardinality $q+1$, with the elements of $V_i$ written $1_i, 2_i, \dots, (q+1)_i$.
Write
\[
  \sigma_i = V_i \setminus \{(q+1)_i\}, \qquad
  \sigma_i' = V_i \setminus \{1_i\}
\]
for two facets of $K^i = \partial\Delta^q$. The Baskakov cocycle on $K^i$ is the
characteristic function $a_i = \chi_{\sigma_i}$, so $S_{a_i} = \{\sigma_i\}$ and
$|S_{a_i}| = 1$. Choosing the distinguished vertex $v_{\sigma_i} = 1_i$ in
Construction~\ref{cons:GL1} gives
$\sigma_i \cap \sigma_i' = \sigma_i \setminus \{1_i\} = V_i \setminus \{1_i,(q+1)_i\}$,
so $P_{a_i} = P_{\sigma_i} = \{\sigma_i'\}$ and $|P_{a_i}| = 1$: the unique
partner facet is $\sigma_i'$.

\begin{definition}
\label{def:Kq}
The simplicial complex $K = K(q)$ on the vertex set $V$ is
\begin{equation}
\label{eq:Kq-def}
  K(q) = \sd_{\tau^{12}}\, \sd_{\tau^{23}}\,(K^1 \ast K^2 \ast K^3),
\end{equation}
where $\tau^{12} = \sigma_1 \cup \sigma_2'$ and $\tau^{23} = \sigma_2 \cup \sigma_3'$
are the two GL star deletions of Construction~\ref{cons:GL1} for $n=3$
(the pairs $(1,2)$ and $(2,3)$). Equivalently, $K(q)$ is the simplicial complex
on $V$ whose minimal non-faces are the three full vertex sets $V_1,V_2,V_3$
(cardinality $q+1$) together with $\tau^{12}$ and $\tau^{23}$ (cardinality $2q$);
in particular $\nu(K) = q+1$.
\end{definition}

In particular, $m = 3(q+1)$, and 
$\nu(K) = \min(q+1, 2q) = q+1$ for every $q \ge 1$.

\begin{lemma}
\label{lem:facets}
For every $q \ge 1$, the facets of $K(q)$ are of two types:
\begin{itemize}
\item \emph{Type A} (cardinality $3q$): the facets $F_1\cup F_2\cup F_3$ of
$K^1 \ast K^2 \ast K^3$ surviving the two star deletions; there are exactly
$(q+1)^3 - 2(q+1) = (q+1)(q^2+2q-1)$ of these.
\item \emph{Type B} (cardinality $3q-1$): a single facet,
$F_B = V \setminus \{(q+1)_1, 1_2, (q+1)_2, 1_3\}$.
\end{itemize}
Consequently
\begin{equation}
\label{eq:r-formula}
  r = r(q) = (q+1)^3 - 2(q+1) + 1 = q^3 + 3q^2 + q .
\end{equation}
\end{lemma}

\begin{proof}
Facets of $K$ are the complements $F = V\setminus W$ of the minimal transversals
$W$ of the minimal-non-face hypergraph $\{V_1,V_2,V_3,\tau^{12},\tau^{23}\}$;
type-A facets ($|F|=3q$) correspond to transversals with $|W|=3$, type-B facets
($|F|=3q-1$) to those with $|W|=4$.

\emph{Type A.} Each $K^i$ has $q+1$ facets, so $K^1\ast K^2\ast K^3$ has
$(q+1)^3$ facets $F=F_1\cup F_2\cup F_3$ of size $3q$. Such a facet contains
$\tau^{12}=\sigma_1\cup\sigma_2'$ iff $F_1=\sigma_1$ and $F_2=\sigma_2'$
(with $F_3$ arbitrary: $q+1$ facets), and contains
$\tau^{23}=\sigma_2\cup\sigma_3'$ iff $F_2=\sigma_2$ and $F_3=\sigma_3'$
(with $F_1$ arbitrary: $q+1$ facets); no facet contains both, since that would
force $F_2=\sigma_2'=\sigma_2$, whereas
$\sigma_2'=V_2\setminus\{1_2\}\ne V_2\setminus\{(q+1)_2\}=\sigma_2$. Hence
$(q+1)^3-2(q+1)$ join facets survive.

\emph{Type B.} Let $W$ be a minimal transversal with $|W|=4$; it meets each
$V_i$, so exactly one block $W_i=W\cap V_i$ has size $2$. The only minimal
non-faces meeting $V_1$ are $V_1$ and $\tau^{12}$, and the only ones meeting
$V_3$ are $V_3$ and $\tau^{23}$; since a single hyperedge can witness only one
vertex of $W$, the cases $|W_1|=2$ and $|W_3|=2$ are impossible. So $|W_2|=2$;
write $W=\{u,v,v',w\}$ with $u\in V_1$, $\{v,v'\}\subset V_2$, $w\in V_3$.
Minimality forces $u,v$ to be witnessed by $\tau^{12}=\sigma_1\cup\sigma_2'$ and
$v',w$ by $\tau^{23}=\sigma_2\cup\sigma_3'$, whence $u=(q+1)_1$ (the vertex of
$V_1$ outside $\sigma_1$), $v=1_2$, $v'=(q+1)_2$, $w=1_3$. This gives the unique
transversal $W=\{(q+1)_1,1_2,(q+1)_2,1_3\}$, i.e.\ the single type-B facet $F_B$.
\end{proof}

\begin{lemma}
\label{lem:Kq-nondeg}
For every $q\ge 1$ and every $Z\subseteq[m]$ containing a minimal non-face of
$K(q)$, one has $|Z|+\rho_Z\ge 2\nu=2(q+1)$. Hence the non-degeneracy
hypothesis of Proposition~\ref{prop:sing} is satisfied, and
\[
\dim\Sing\bigl(V(\Phi_K)\bigr)=m+r-2\nu=q^3+3q^2+2q+1 .
\]
\end{lemma}

\begin{proof}
Let $\MM(Z)$ be the set of minimal non-faces of $K(q)$ contained in $Z$
(nonempty), so $\rho_Z=\bigl|\bigcap_{\tau\in\MM(Z)}\tau\bigr|$.  The
minimal non-faces are the pairwise disjoint blocks $V_1,V_2,V_3$ (of size
$\nu=q+1$) and the deleted faces $\tau^{12}=\sigma_1\cup\sigma_2'\subseteq
V_1\cup V_2$ and $\tau^{23}=\sigma_2\cup\sigma_3'\subseteq V_2\cup V_3$ (of size
$2q$); note $V_1\cap\tau^{23}=V_3\cap\tau^{12}=\varnothing$ and
$\tau^{12}\cap\tau^{23}=V_2\setminus\{1_2,(q+1)_2\}$.

If $|\MM(Z)|\le 2$, the inequality is Remark~\ref{rem:sing-hypothesis}.
Suppose $|\MM(Z)|\ge 3$.  If two members of $\MM(Z)$ are disjoint,
then $\rho_Z=0$, while any two disjoint minimal non-faces have total size at
least $2\nu$, so $|Z|\ge 2\nu$.  Otherwise the members pairwise intersect;
since the blocks are pairwise disjoint, $V_1\cap\tau^{23}=V_3\cap\tau^{12}
=\varnothing$, and $\tau^{12}\cap\tau^{23}\ne\varnothing$ forces $q\ge 2$, the
only such family is $\MM(Z)=\{V_2,\tau^{12},\tau^{23}\}$.  Then
$Z\supseteq V_2\cup\tau^{12}\cup\tau^{23}=\sigma_1\cup V_2\cup\sigma_3'$, 
whence $|Z|\ge 3q+1>2\nu$.  In all cases $|Z|+\rho_Z\ge 2\nu$.

The value $2\nu$ is attained (at $Z=V_i$, where $\rho_Z=\nu$, and at
$Z=V_i\cup V_j$, where $\rho_Z=0$), so $\min_Z(|Z|+\rho_Z)=2\nu$ and
Proposition~\ref{prop:sing} gives the stated dimension.
\end{proof}

\begin{prop}
\label{prop:K-homotopy}
For every $q \ge 1$, $|K(q)| \simeq S^{3q-2}$. In particular
$\widetilde{H}^*(K(q);\Z)$ is free of rank $1$, concentrated in degree $3q-2$.
\end{prop}

\begin{proof}
Put $S = K^1\ast K^2\ast K^3$. Since $K^i=\partial\Delta^q\simeq S^{q-1}$, we have
$|S|\cong S^{3q-1}$, and by Definition~\ref{def:Kq}
$|K| = |S|\setminus\bigl(\operatorname{ost}_S\tau^{12}\cup\operatorname{ost}_S\tau^{23}\bigr)$.
No facet $F=F_1\cup F_2\cup F_3$ of $S$ contains both $\tau^{12}$ and $\tau^{23}$:
the first requires $F_2=\sigma_2'$, the second $F_2=\sigma_2$, and
$\sigma_2'\ne\sigma_2$. Hence the two open stars are disjoint open PL balls, and
$|K|$ is $S^{3q-1}$ with two distinct points deleted, which is homotopy
equivalent to $S^{3q-2}$.
\end{proof}

\begin{corollary}
\label{cor:cohomology}
For every $q \ge 2$, the moment-angle complex $\ZK$ associated with 
$K = K(q)$ satisfies
\[
  H^{2q}(\ZK;\Z) = 0, \qquad
  H^{2q+1}(\ZK;\Z) \cong \Z^3, \qquad
  H^{6q+2}(\ZK;\Z) \cong \Z .
 \]
In particular, $\ZK$---and hence $F_K \simeq \ZK$---is $2q$-connected.
\end{corollary}

\begin{proof}
By Hochster's formula~\cite{Ho77}, 
$H^n(\ZK;\Z) \cong \bigoplus_{J\subseteq V}\widetilde{H}^{n-|J|-1}(K_J;\Z)$. 
Throughout, a summand can be nonzero only if $K_J$ is non-contractible, 
which forces $J$ to contain a minimal non-face of $K$; recall these are 
the three sets $V_i$ (cardinality $q+1$) and the $2q$ deleted faces 
$\tau$ (cardinality $2q$).

\smallskip
\noindent\emph{The group $H^{2q+1}$.}  Here the summands are 
$\widetilde{H}^{2q-|J|}(K_J)$, so $|J|\le 2q$.  If $J\supseteq\tau$ for 
a deleted face, then $|J|=2q$ and $J=\tau$, giving 
$K_\tau = \partial\Delta^{2q-1}\simeq S^{2q-2}$ and the contribution 
$\widetilde{H}^{0}(S^{2q-2}) = 0$.  Otherwise $J$ contains a single 
$V_i$ (two would need $\ge 2q+2$ vertices), so $J = V_i\sqcup J'$ with 
$J'\subseteq\bigcup_{l\ne i}V_l$ of cardinality $\le q-1$; then $V_i$ is 
the only minimal non-face in $J$ and $K_J = \partial\Delta^{V_i}\ast
\Delta^{J'}$, contractible unless $J' = \varnothing$.  Thus only 
$J = V_1, V_2, V_3$ contribute, each $\widetilde{H}^{q-1}
(\partial\Delta^q)\cong\Z$, whence $H^{2q+1}(\ZK;\Z)\cong\Z^3$.

\smallskip
\noindent\emph{The group $H^{2q}$.}  The same analysis at $n=2q$ gives 
summands $\widetilde{H}^{2q-|J|-1}(K_J)$ with $|J|\le 2q-1$; no $J$ 
contains a $\tau$, and the single-$V_i$ case $J = V_i\sqcup J'$ has 
$K_J=\partial\Delta^{V_i}\ast\Delta^{J'}$ contractible for 
$J'\ne\varnothing$, while $J=V_i$ contributes 
$\widetilde{H}^{q-2}(\partial\Delta^q)=0$.  Hence $H^{2q}(\ZK;\Z)=0$.

\smallskip
\noindent\emph{The group $H^{6q+2}$.}  The summands are 
$\widetilde{H}^{6q+1-|J|}(K_J)$.  As $\widetilde{H}^d(K_J)=0$ for 
$d > |J|-1$, a nonzero summand needs $6q+1-|J|\le |J|-1$, i.e.\ 
$|J|\ge 3q+1$; with $|J|\le m = 3q+3$, only $|J|\in\{3q+1,3q+2,3q+3\}$ 
arise.
\begin{itemize}
\item $|J| = 3q+3$, i.e., $J = V$: 
by Proposition~\ref{prop:K-homotopy}, $\widetilde H^{3q-2}(K)\cong\Z$, so 
$J = V$ contributes $\Z$.
\item $|J| = 3q+1$, i.e., $J = V\setminus\{v,w\}$: the summand 
$\widetilde{H}^{3q}(K_J)$ sits in the top degree $|J|-1$.  A nonzero 
class there requires a $3q$-dimensional face spanning all of $J$; but 
then $J\in K$ and $K_J=\Delta^{J}$ is contractible, while if $J\notin K$ 
there are no such faces.  Either way, $\widetilde{H}^{3q}(K_J)=0$.
\item $|J| = 3q+2$, i.e.\ $J = V\setminus\{v\}$: write 
$K = K_{V\setminus v}\cup\overline{\operatorname{st}}_K(v)$, where the closed star 
$\overline{\operatorname{st}}_K(v)$ is a cone (hence contractible) and 
$K_{V\setminus v}\cap\overline{\operatorname{st}}_K(v) = \operatorname{lk}_K(v)$.  
The reduced Mayer--Vietoris sequence, together with $\widetilde{H}^{3q-1}(K) = 
\widetilde{H}^{3q}(K) = 0$ (Proposition~\ref{prop:K-homotopy}), yields 
$\widetilde{H}^{3q-1}(K_{V\setminus v})\cong\widetilde{H}^{3q-1}
(\operatorname{lk}_K v)$.  Every facet of $K$ has at most $3q$ vertices, so 
$\dim\operatorname{lk}_K(v)\le 3q-2$ and this group vanishes.
\end{itemize}
Hence $J = V$ is the only contributor and $H^{6q+2}(\ZK;\Z)\cong \Z$.

\smallskip
\noindent\emph{Connectivity.}  By Theorem~\ref{thm:ZZ} 
($\nu(K)=q+1$), $\ZK$ is $2q$-connected. The groups computed above 
show this is sharp, since $H^{2q}(\ZK;\Z)=0$ while 
$H^{2q+1}(\ZK;\Z)\cong\Z^3\ne 0$.
\end{proof}

\subsection{The polynomial $\Phi_K$ and its Milnor fiber}
\label{subsec:Phi-Kq-Milnor}

We can now state the main theorem on the family.  The 
combinatorial bookkeeping ($m$, $\nu$, $r$, $N$) is carried by 
Definition~\ref{def:Kq} and Lemma \ref{lem:facets}; 
the homotopy and cohomology results by 
Proposition~\ref{prop:K-homotopy} and Corollary~\ref{cor:cohomology}.  

\begin{theorem}
\label{thm:family}
For each $q \ge 2$, let $K = K(q)$ be the simplicial complex of 
Definition~\ref{def:Kq}, and let $\Phi_K$ be the associated 
weighted-homogeneous polynomial of Corollary~\ref{cor:FdB-ZK}, 
with $m = 3(q+1)$, $r = q^3 + 3q^2 + q$, and $N = m+r$.  Then:
\begin{enumerate}[label=(\roman*), font=\upshape, itemsep=2pt]
\item \label{p1}
$\Phi_K \colon \C^N \to \C$ is weighted-homogeneous of degree $5$, 
with $w(x_j) = 1$ for $j \in [m]$, $w(y_F) = 2$ for the 
$y$-variables corresponding to type-A facets, and $w(y_F) = 1$ 
for those corresponding to type-B facets.
\item \label{p2}
The Milnor fiber $F_K = \Phi_K^{-1}(1) \simeq \ZK$ is 
$2q$-connected, and $H^*(F_K;\Q)$ contains a non-trivial triple 
Massey product in degree $6q+2$.
\item \label{p3}
$F_K$ has exact formality degree $6q$: it is $6q$-formal but not 
$(6q+1)$-formal.
\end{enumerate}
In particular, Theorem~\ref{thm:main} is realized for every 
$k \ge 1$ and $n = 3$ by taking $q = k+1$.
\end{theorem}

\begin{proof}
\ref{p1} By Lemma~\ref{lem:facets}, the smallest facet of $K$ 
has cardinality $3q - 1$, so the polynomial 
$\Phi_K = \sum_F y_F\, g_F$ of Corollary~\ref{cor:FdB-ZK} has 
degree $1 + (m - (3q-1)) = 5$, with $w(x_j) = 1$, and 
$w(y_F) = 5 - \deg g_F = 5 - (m - |F|)$.  For a type-A facet 
$|F| = 3q$, so $w(y_F) = 5 - 3 = 2$.  For a type-B facet 
$|F| = 3q - 1$, so $w(y_F) = 5 - 4 = 1$.

 \smallskip
 \ref{p2} By Corollary~\ref{cor:cohomology}, $\ZK$ is
 $2q$-connected.  The Grbi\'c--Linton join theorem
 (Theorem~\ref{thm:GL-joins}) applied to the Baskakov-cocycle
 construction of $K(q)$ produces a non-trivial triple Massey product
 $\langle\alpha_1,\alpha_2,\alpha_3\rangle$ in $H^{6q+2}(\ZK;\Q)$,
 where $\alpha_i\in H^{2q+1}(\ZK;\Q)$ is the Hochster class of the
 generator of $\widetilde{H}^{q-1}(K^i;\Q)$.  By
 Corollary~\ref{cor:FdB-ZK}, $F_K\simeq\ZK$, and both the connectivity
 and the Massey-product class transfer.

 \smallskip
 \ref{p3} The triple Massey product
 $\langle\alpha_1,\alpha_2,\alpha_3\rangle$ of~\ref{p2} lands in degree
 $L = 6q+2$, so by Lemma~\ref{lem:massey-formal} its non-triviality
 obstructs $(6q+1)$-formality.  It is the lowest such obstruction: as
 $F_K$ is $2q$-connected, its minimal model has generators only in
 degrees $\ge 2q+1$, so through degree $6q$ the differential is
 quadratic in closed generators---a generator of degree $\le 6q$
 differentiates to a product of two generators of degrees in
 $[2q+1,4q]$, all closed since $2(2q+1) > 4q+1$---whence the
 $6q$-minimal model agrees with that of $(H^*(F_K;\Q),0)$.  Thus $F_K$
 is $6q$-formal but not $(6q+1)$-formal, of exact formality degree
 $6q$.
\end{proof}

\begin{example}[$K^i = \partial\Delta^2$, $n = 3$: the first GL example]
\label{ex:k2n3}
Specialize to $q = 2$: take $K^i = \partial\Delta^2$ on the disjoint vertex sets
$V_1=\{a,b,c\}$, $V_2=\{d,e,f\}$, $V_3=\{g,h,i\}$ (so $1_i,2_i,3_i$ are
$a,b,c$, etc.). With $a_i = \chi_{\sigma_i}$ and $v_{\sigma_i}=1_i$, the partner
facets are $\sigma_1'=\{b,c\}$, $\sigma_2'=\{e,f\}$, $\sigma_3'=\{h,i\}$, so
$\tau^{12}=\{a,b,e,f\}$ and $\tau^{23}=\{d,e,h,i\}$, and Definition~\ref{def:Kq}
gives
\[
K(2) = \sd_{\{a,b,e,f\}}\, \sd_{\{d,e,h,i\}}\,(K^1 \ast K^2 \ast K^3)
\]
on $m = 9$ vertices, with minimal non-faces $\{a,b,c\}$, $\{d,e,f\}$,
$\{g,h,i\}$ and the two deleted faces $\{a,b,e,f\}$, $\{d,e,h,i\}$; hence
$\nu(K) = 3$. Lemma~\ref{lem:facets} gives $r = 22$: twenty-one type-A facets of
size $6$ and the single type-B facet $F_B = \{a,b,e,h,i\}$ of size $5$. The
polynomial $\Phi_K\colon \C^{31}\to\C$ is weighted-homogeneous of degree $5$,
with $21$ degree-$3$ monomials (of $y$-weight $2$) and one degree-$4$ monomial
(of $y$-weight $1$).

The Poincar\'e polynomial of $F_K\simeq\ZK$, computed by \texttt{Macaulay2}, is
\begin{equation}
\label{eq:poin-fk}
P_{F_K}(t) = 1 + 3t^5 + 2t^7 + 4t^8 + t^{10} + 2t^{12} + 4t^{13} + t^{14}.
\end{equation}
Since $\ZK$ is $4$-connected (Theorem~\ref{thm:ZZ}), the polynomial~\eqref{eq:poin-fk}
shows $H^5(\ZK;\Q)\cong\Q^3\neq 0$, so $\ZK$ is not $5$-connected and the connectivity
bound of Theorem~\ref{thm:ZZ} is sharp. (By universal coefficients $H^5(\ZK;\Z)$ is
torsion-free, hence $\cong\Z^3$.) The triple Massey product lives in $H^{14}(\ZK;\Q)\cong\Q$.
By Proposition~\ref{prop:sing}, 
$\dim\Sing(V(\Phi_K)) = m+r-2\nu = 9+22-6 = 25$, of
codimension $5$; correspondingly the Kato--Matsumoto bound
$N-s-2 = 31-25-2 = 4$ equals the connectivity of $F_K$.
\end{example}

\subsection{Further remarks}
\label{subsec:further-remarks}

We close the section with a few remarks on the invariants of the family
$K(q)$, its higher-$n$ analogues, and the hypotheses behind the construction.

\begin{remark}
\label{rem:family-m2}
For the family $K(q)$ one has the closed forms
\begin{equation}
\label{eq:mrNnu}
m = 3(q+1), \quad r = q^3+3q^2+q, \quad N = q^3+3q^2+4q+3, \quad
\nu = q+1,
\end{equation}
and, by Lemma~\ref{lem:Kq-nondeg}, the non-degeneracy condition 
of Proposition~\ref{prop:sing} holds, so 
\begin{equation}
\label{eq:sing-phik}
s = \dim\Sing\bigl(V(\Phi_K)\bigr) = m+r-2\nu = q^3+3q^2+2q+1 .
\end{equation}
Consequently, the Kato--Matsumoto bound $N-s-2 = 2q$ equals the connectivity 
$2q$ of $F_K$ throughout the family. The facet count $r$ and the singular-locus 
dimension $s$ were verified directly in \texttt{Macaulay2} for $2\le q\le 5$.
\end{remark}

\begin{remark}
\label{rem:n-family}
For fixed $q=2$ (so $K^i=\partial\Delta^2$, $k=1$) and varying $n\ge 3$, the
Baskakov cocycle again has $|S_{a_i}|=1$ and $|P_{a_i}|=1$, so each of the
$\binom{n}{2}-1$ pairs $(i,j)$ with $(i,j)\ne(1,n)$ contributes a single GL
star deletion, for $\binom{n}{2}-1$ deletions in all. 
The simplicial complex $K$ has $m = 3n$ vertices, and the
minimal-non-face count in Step~2 of the proof of
Theorem~\ref{thm:main} is uniform in $n$, so it gives
$\nu(K) = k+2 = 3$ for every $n\ge 3$: the minimal non-faces of $K$
are the $n$ triangles $[3]_i$, of cardinality~$3$, together with the
star-deleted faces, of cardinality $2k+2 = 4$.  By
Theorem~\ref{thm:ZZ}, $\ZK$ is therefore $4$-connected for every
$n\ge 3$, and the estimate of \S\ref{subsec:Hochster-conn} makes this
sharp: $\nu(K)=3$ forces $\widetilde{H}^i(\ZK;\Z)=0$ for $i\le 4$,
while the minimal non-faces $[3]_i$ contribute a nonzero class in
$H^{2\nu(K)-1}=H^5(\ZK;\Z)$.  For $n=3$ this is
Example~\ref{ex:k2n3}.  
The $n$-fold Massey product lands in $H^{4n+2}(F_K;\Z)$ by the 
degree formula $2n(k+1)+2$ with $k = 1$. The number of facets 
$r = r(n)$, and hence the ambient dimension $N = 3n + r(n)$, 
depends on the combinatorial geometry of the deletions and is 
determined by direct enumeration; for $n = 3$, 
Example~\ref{ex:k2n3} gives $r = 22$.
\end{remark}

\begin{remark}
\label{rem:wh-necessary}
The passage to weighted-homogeneous polynomials in 
Theorem~\ref{thm:family} is unavoidable.  By \S\ref{subsec:FdB}, 
$\Phi_K$ is homogeneous if and only if $K$ is pure, i.e., all 
facets have the same cardinality.  By Lemma~\ref{lem:facets}, 
$K(q)$ has facets of two distinct sizes ($3q$ and $3q - 1$) for 
every $q \ge 1$, so $K(q)$ is never pure and $\Phi_K$ is 
genuinely weighted-homogeneous but not homogeneous in this 
entire family.  Whether a homogeneous polynomial can have a 
highly connected non-formal Milnor fiber at all is 
Problem~\ref{prob:homogeneous} below.
\end{remark}

\section{Open problems}
\label{sec:questions}

\begin{problem}
\label{prob:homogeneous}
Is there a version of Theorem~\ref{thm:main} where the polynomial
$f$ is homogeneous?  By Remark~\ref{rem:wh-necessary}, any such
example cannot arise from the GL construction applied to
$K^i = \partial\Delta^{k+1}$ with $k\ge 1$, so a genuinely
different approach would be needed.
\end{problem}

\begin{problem}
\label{prob:edge}
The Milnor fibers constructed here realize the Grbi\'c--Linton join
construction (Construction~\ref{cons:GL1}).  Are the edge-stretched families of
Construction~\ref{cons:GL2} and Theorem~\ref{thm:GL-edge} likewise realizable as
Milnor fibers of weighted-homogeneous polynomials, in the sense of
Theorem~\ref{thm:FdB}?
\end{problem}

\begin{problem}
\label{prob:arrangements}
Can one find hyperplane arrangement polynomials $Q(\A)$ (not just
weighted-homo\-geneous polynomials via \cite{FdB}) whose Milnor fibers $F(\A)$
are non-formal with non-formality witnessed directly by non-trivial
Massey products in $H^*(F(\A);\Q)$?  The Grbi\'c--Linton framework
may be relevant here if one can identify $F(\A)$ as a moment-angle
complex, or find a map from a non-formal $\ZK$ to $F(\A)$ that
is surjective on cohomology.
\end{problem}

\begin{problem}
\label{prob:comparison}
Determine the precise relationship between the non-formality of
$F(\A)$ detected via the tangent cone / characteristic variety
method (as in \cite{Su-mfnf}) and the non-formality detected
via non-trivial Massey products.  Is every non-formal Milnor fiber
of a hyperplane arrangement detected by some Massey product?
\end{problem}

\begin{problem}
\label{prob:KM}
For the $n = 3$ family of Theorem~\ref{thm:family}, the 
Kato--Matsumoto bound is sharp: Proposition~\ref{prop:sing} gives 
$s = \dim \Sing(\Phi_K) = m + r - 2\nu = q^3 + 3q^2 + 2q + 1$, 
its non-degeneracy condition holding here (Lemma~\ref{lem:Kq-nondeg}); 
therefore $N - s - 2 = 2\nu - 2 = 2q$, exactly the connectivity of 
$F_K \simeq \ZK$ (Corollary~\ref{cor:cohomology}).  Does this 
sharpness persist for the general-$n$ construction of 
Theorem~\ref{thm:main}?  There $\nu(K) = k+2$ and the Milnor fiber is 
still $(2k+2)$-connected (Step~2 of its proof and 
Theorem~\ref{thm:ZZ}), so by Proposition~\ref{prop:sing} the bound is 
sharp precisely when $s = m + r - 2(k+2)$, i.e.\ when the 
non-degeneracy condition $|Z| + \rho_Z \ge 2\nu$ holds for every 
non-face $Z$.  That condition is not established for general $n$: by 
Remark~\ref{rem:sing-hypothesis} it could fail only for a set $Z$ 
meeting three or more minimal non-faces with empty common 
intersection, and whether such a configuration arises among the 
deleted faces is not understood.  Moreover the facet count $r = r(n)$, 
hence $N$, is known only by direct enumeration 
(Remark~\ref{rem:n-family}); determining $s$, and whether $N - s - 2$ 
continues to match $2k+2$, remains open.
\end{problem}

\begin{problem}
\label{prob:quaternionic}
The complexes $K$ produced in Theorem~\ref{thm:main} have minimal
non-faces of cardinality $\nu(K) = k+2$ (Step~2 of its proof).  For any
such $K$, the fibration
$\ZK(D^4,S^3)\to\ZK(\mathbb{HP}^\infty,*)\to(\mathbb{HP}^\infty)^m$ has
base cohomology $\Z[v_1,\dots,v_m]$ with $\deg v_i = 4$, and the lowest
Stanley--Reisner relation, of degree $4\nu(K)$, transgresses to
$H^{4\nu(K)-1}$ of the fiber; hence $\ZK(D^4,S^3)$ is
$\bigl(4\nu(K)-2\bigr) = (4k+6)$-connected.  Does there exist a
holomorphic $f\colon\C^N\to\C$ whose Milnor fiber is homotopy equivalent
to $\ZK(D^4,S^3)$?  An affirmative answer would raise the connectivity
in Theorem~\ref{thm:main} from $2k+2$ to $4k+6$.
\end{problem}

\newcommand{\arxiv}[1]
{\texttt{\href{http://arxiv.org/abs/#1}{arXiv:#1}}}
\newcommand{\arxi}[1]
{\texttt{\href{http://arxiv.org/abs/#1}{arxiv:}}
\texttt{\href{http://arxiv.org/abs/#1}{#1}}}
\newcommand{\arxx}[2]
{\texttt{\href{http://arxiv.org/abs/#1.#2}{arxiv:#1.}}%
\texttt{\href{http://arxiv.org/abs/#1.#2}{#2}}}
\newcommand{\doi}[1]
{\texttt{\href{http://dx.doi.org/#1}{doi:\nolinkurl{#1}}}}
\renewcommand{\MR}[1]
{\href{http://www.ams.org/mathscinet-getitem?mr=#1}{MR#1}}

\end{document}